\documentclass[11pt]{article}
\usepackage[nohead,margin=1.0in]{geometry}
\usepackage{amssymb, amsmath, amsthm,amsfonts}
\usepackage{amsmath}
\usepackage{graphicx,epsfig}
\usepackage[tight]{subfigure}
\usepackage{cite}
\usepackage{times}
\usepackage{float}
\graphicspath{{./pics/}}
\usepackage{color}
\usepackage{epstopdf}
\usepackage{mathrsfs}
\usepackage{mathtools}
\usepackage{threeparttable}
\usepackage{algorithm,algorithmic}
\usepackage{booktabs}
\usepackage{verbatim}
\usepackage{bm}

\usepackage[colorlinks,linktocpage,linkcolor=blue]{hyperref}

\numberwithin{equation}{section}
\usepackage{array}
\newtheorem{thm}{Theorem}[section]
\newtheorem{lem}{Lemma}[section]
\newtheorem{remark}{Remark}[section]
\newtheorem{example}{Example}[section]
\newtheorem{prop}{Proposition}[section]

\newtheorem{assum}{Assumption}[section]

\allowdisplaybreaks

\def\al{\alpha}

\def\Om{\Omega}

\def\Rbb{\mathbb{R}}

\def\II{{(\Omega)}}

\def\II{(\Omega)}

\renewcommand{\d}{\mathrm{d}}

\title{Numerical Approximation and Analysis of the Inverse Robin Problem Using the Kohn-Vogelius Method\thanks{E.B. was supported by the EPSRC grants EP/T033126/1 and EP/V050400/1. For the purpose of open access, the author has applied a Creative Commons Attribution (CC BY) licence to any Author Accepted Manuscript version arising. The work of B. J. is supported by Hong Kong RGC General Research Fund (Project
14306423 and Project 14306824), and a start-up fund from The Chinese University of Hong Kong. The work of Z. Z. is supported by by National Natural
Science Foundation of China (Project 12422117) and Hong Kong Research
Grants Council (15302323), and an internal grant of Hong Kong Polytechnic
University (Project ID: P0038888, Work Programme: ZVX3).
}}

\author{Erik Burman\thanks{Department of Mathematics, University College London, London WC1E 6BT, UK (\texttt{e.burman@ucl.ac.uk})}\and
Siyu Cen\thanks{Department of Applied Mathematics, The Hong Kong Polytechnic University, Kowloon, Hong Kong, P.R. China (\texttt{siyu.cen2021@connect.polyu.hk, zhizhou@polyu.edu.hk})}
\and Bangti Jin\thanks{Department of Mathematics, The Chinese University of Hong Kong, Shatin, New Territories, Hong Kong, P.R. China (\texttt{bangti.jin@gmail.com, b.jin@cuhk.edu.hk}).}
\and Zhi Zhou\footnotemark[3]}

\begin{document}

\maketitle

\begin{abstract}
In this work, we numerically investigate the inverse Robin problem of recovering a piecewise constant Robin coefficient in an elliptic or parabolic problem from the Cauchy data on a part of the boundary, a problem that commonly arises in applications such as non-destructive corrosion detection. We employ a Kohn-Vogelius type variational functional for the regularized reconstruction, and discretize the resulting optimization problem using the Galerkin finite element method on a graded mesh. We establish rigorous error estimates on the recovered Robin coefficient in terms of the mesh size, temporal step size and noise level. This is achieved by combining the approximation error of the direct problem, \textit{a priori} estimates on the functional, and suitable conditional stability estimates of the continuous inverse problem. We present several numerical experiments to illustrate the approach and to complement the theoretical findings.

\end{abstract}

\section{Introduction}\label{sec:intro}
This paper is concerned with the inverse problem of recovering a Robin coefficient $q(x)$ in elliptic or parabolic problems from the boundary measurement. Let $\Om\subseteq \Rbb^2$ be a convex polyhedral domain with a boundary $\partial\Omega$. Let $\Gamma_i$ and $\Gamma_a$ be disjoint parts of the boundary $\partial\Omega$ satisfying $\Gamma_i\cup\Gamma_a=\partial\Omega$. The boundary $\Gamma_i$ is inaccessible to take measurement. Consider the following boundary value problem for $u$:
\begin{equation}\label{eqn:problem_NR}
    \left\{
        \begin{aligned}
            -\Delta u & =0, && \mbox{in }\Om, \\
            \partial_n u+qu & =0, && \mbox{on } \Gamma_i,\\
            \partial_nu & =g, && \mbox{on } \Gamma_a,
        \end{aligned}
    \right.
\end{equation}
where $n$ denotes the unit outward normal vector to the boundary $\partial\Omega$, and $g$ denotes the given boundary data. The solution $u$ to problem \eqref{eqn:problem_NR} is denoted by $u(q)$ to indicate its dependence on the Robin coefficient $q$. The {inverse Robin problem} is to recover the Robin coefficient $q^\dag$ on the inaccessible part $\Gamma_i$ from a noisy version of the Dirichlet data $f=u(q^\dag)|_{\Gamma_a'} $ on a connected subset $\Gamma_a' \subset\Gamma_a$.

The inverse problem arises in several practical applications. In thermal engineering, problem \eqref{eqn:problem_NR} can describe steady-state heat transfer in which convective conduction occurs at the interface $\Gamma_i$ between the conducting body $\Omega$ and its surroundings; the variable $u$ represents the temperature, $g$ is heat flux and $q$ is closely related to physical and chemical properties of the material at the interface $\Gamma_i$. Then the inverse Robin problem is to determine the thermal convection property on $\Gamma_i$ from indirect measurement. It arises also in non-destructive testing of corrosion in an electrostatic conductor \cite{Inglese:1997}. The Robin boundary condition arises when considering a thin oscillating coating surrounding a homogeneous background medium such that the thickness of the layer and the wavelength of the oscillations tend simultaneously to 0 \cite{ButtazzoKohn:1987}.  It arises also naturally in determining the contact resistivity of planar electronic
devices, e.g., metal oxide semiconductor field-effect transistors \cite{Fang:1992}.

Several numerical approaches for reconstructing the Robin coefficient $q$ from boundary measurement have been developed \cite{Inglese:1997,Chaabane:2004,Jaoua:2008,Jin:2007,Jin:2009,Jin:2010,Chaabane:2012,Harrach:2021,ChenJiangZou:2022,Rasmussen:2023}. Inglese \cite{Inglese:1997} proposed to decouple the problem by first solving the Cauchy problem of Laplace equation (with the Cauchy data on $\Gamma_a'$), and then using the algebraic relation $q=\partial_n u(q)/ u(q)$ to handle the nonlinearity; See  also \cite{Jaoua:2008} for a similar analytic extension approach. 
Several algorithms  based on the variational approach \cite{EnglHankeNeubauer:1996,ItoJin:2015} have been developed. Jin and Zou \cite{Jin:2010} employed the output least-squares functional with an $L^2(\Gamma_i)$ penalty to recover a smooth $q(x)$, and analyzed the convergence behavior of the approximation as the noise level $\delta$ and mesh size  $h\rightarrow 0$. Later they \cite{Jin:2009} studied reconstructing a piecewise constant $q(x)$ with the Modica-Mortola functional and discussed the convergence of the discrete approximation. Xu and Zou \cite{XuZou:2015sicon} developed an adaptive finite element method (FEM), and proved the convergence of the adaptive method. Chaabane et al employed the Kohn-Vogelius functional \cite{KohnVogelius:1987} in \cite{Chaabane:2004} and \cite{Chaabane:2012} to recover smooth and piecewise constant $q(x)$, respectively, and reported excellent stability of the method with respect to noise. In practice, the variational formulations are discretized using the Galerkin FEM. This inevitably introduces additional errors, which can affect the quality of the reconstruction. However, the above studies have either ignored the discretization step or only considered the convergence of the FEM approximation, but not the convergence rate. One goal of this work is precisely to fill the gap.

The numerical reconstruction of the parabolic inverse Robin problem has been extensively studied in the literature \cite{Slodivcka:2010,JinLu:2012,Hao:2013,Zhuo:2020,CaoLesnicLiu:2020,Pyatkov:2023}. Specifically, \cite{JinLu:2012,Zhuo:2020} addressed the recovery of a space-time dependent Robin coefficient $q(x,t)$ from measurements on $\Gamma_a \times (0,T)$ using a penalized least-squares formulation. The convergence of the FEM approximation was analyzed in \cite{JinLu:2012}, though no error bounds were established. In \cite{CaoLesnicLiu:2020}, Cao et al. extended the problem by simultaneously reconstructing $q$ and either the heat flux $g$ or the initial data $u_0$. To the best of our knowledge, the uniqueness and stability of determining a general space-time dependent $q(x,t)$ remain largely unexplored, except for the one-dimensional case \cite{Slodivcka:2010} and the separable case \cite{Pyatkov:2023}.

In this work, we revisit the numerical reconstruction of a piecewise constant Robin coefficient $q(x)$ using the Kohn-Vogelius functional discretized by the Galerkin FEM for both elliptic and parabolic problems, and contribute to the following three aspects. First, in the elliptic case, we derive error bounds in terms of the noise level $\delta$ and mesh size $h$, in which $q$ may or may not have a known partition; see Theorem \ref{thm:error_estimate} for the precise statement. The error analysis combines the conditional stability result for the inverse Robin problem (\cite[Theorem 2.4]{Sincich:2007} and \cite[Theorem 4.1]{Hu:2015} for the Lipschitz and H\"{o}lder stability, respectively) with the \textit{a priori} estimate on the discrete functional, and the error bounds are comparable to the stability estimates. Second, we analyze the parabolic inverse Robin problem of recovering a time-independent piecewise constant Robin coefficient $q^\dag(x)$ from the terminal boundary observation $u(x,T)$ on the boundary $\Gamma_a'$. We establish two new stability estimates for large time $T$, complementing the uniqueness result in \cite{Hao:2013}. Third, we develop a novel numerical algorithm based on the Kohn-Vogelius functional and a fully discrete scheme in the parabolic case, and provide rigorous error bounds on the discrete approximation. The analysis employs several technical tools, including numerical analysis of the direct problem with nonsmooth data, decay estimate of the solution, \textit{a priori} estimate on the functional and the continuous stability estimate in the parabolic case.

The rest of the paper is organized as follows. In Section \ref{sec:elliptic}, we discuss the elliptic inverse Robin problem, including the FEM discretization and error bounds.
Then in Section \ref{sec:parabolic}, we extend the analysis to the parabolic case, and derive new stability results and error bounds for the discrete approximation. In Section \ref{sec:numer}, we present several numerical experiments to illustrate distinct features of the approach. Throughout, we use standard notation for Sobolev space $L^2(\Omega)$ and $H^s(\Omega)$. For any set $D$, the notation $(\cdot,\cdot)_{L^2(D)}$ denotes the $L^2(D)$ inner product, with $(\cdot,\cdot)_{L^2(\Omega)}$ abbreviated to $(\cdot,\cdot)$. We denote by $c$ a generic constant not necessarily the same at each occurrence but it is always independent of the noise level $\delta$, the discretization parameters ($h$ and $\tau$) and the penalty parameter $\alpha$.

\section{Inverse Robin problem in the elliptic case}\label{sec:elliptic}

\subsection{Elliptic regularity for mixed boundary value problems}
First we describe a useful regularity result for problem \eqref{eqn:problem_NR}, which plays a crucial role in the error analysis. In addition to problem \eqref{eqn:problem_NR}, we also need the following problem with mixed boundary conditions:
\begin{equation}\label{eqn:problem_DR}
    \left\{\begin{aligned}
            -\Delta u & =0, && \mbox{in }\Om, \\
            \partial_nu+qu & =0, && \mbox{on } \Gamma_i,\\
            \partial_n u & =g, && \mbox{on } \Gamma_a\setminus\Gamma_a',\\
            u & =f, && \mbox{on } \Gamma_a',
        \end{aligned}\right.
\end{equation}
with $f=u(q^\dag)|_{\Gamma_a'}$ being the exact data. Problem \eqref{eqn:problem_DR} will be used in formulating the Kohn-Vogelius functional.
First we precisely specify the admissible set on the Robin coefficient $q$. Throughout, we focus on a piecewise constant Robin coefficient $q^\dag$ of the following two classes.
\begin{itemize}
    \item[(i)] $q^\dag$ is piecewise constant on an \textit{a priori} known partition  $\overline{\Gamma_i}=\bigcup_{j=1}^{N}\overline{\Gamma_{i,j}}$, with connected and pairwise non-overlapping $\Gamma_{i,j}$. The admissible set $\mathcal{A}$ is given by
    \begin{equation}\label{eqn:ad_set_known}
        \mathcal{A}=\bigg\{q(x)=\sum_{j=1}^N q_j\chi_{\Gamma_{i,j}}(x)\,:\, 0<\underline{c}_q\le q_j\le \Bar{c}_q \bigg\}.
    \end{equation}
    \item[(ii)] $q^\dag$ is piecewise constant on an unknown partition  $\overline{\Gamma_i}=\bigcup_{j=1}^{N}\overline{\Gamma_{i,j}}$, with connected and pairwise non-overlapping $\Gamma_{i,j}$. The admissible set $\mathcal{B}$ is given by
    \begin{equation}\label{eqn:ad_set_unknown}
    \mathcal{B}=\left\{q(x):\begin{aligned}
            &\text{there exists a finite partition }\bigcup_{j=1}^{N}\Gamma_{i,j} \text{ of }\Gamma_i \text{ such that } |\Gamma_{i,j}|>c_0 \\
            &\text{and that } q(x)=q_j\in [ \underline{c}_q, \Bar{c}_q]  \text{ on }\Gamma_{i,j},\, j=1,2,\cdots,N
        \end{aligned}\right\},
    \end{equation}
  where $c_0>0$ is a constant and $|\Gamma_{i,j}|$ denotes the Lebesgue measure of $\Gamma_{i,j}$, the number $N$ of partition is not exactly known, but on which an upper bound is given.
\end{itemize}

Now we discuss the solution regularity for problems \eqref{eqn:problem_NR} and \eqref{eqn:problem_DR}. The solutions have weak singularities at the vertices $s_j$ where the types of boundary conditions change. The set of vertices $\{s_j\}_{j=0}^{N+2}$ are given by
\begin{align*}
    &\{s_0,s_N \}= \overline{\Gamma_a}\cap \overline{\Gamma_i}, \quad s_j=\overline{\Gamma_{i,j}}\cap \overline{\Gamma_{i,j+1}},\quad j=1,\dots,N-1,\quad
    \{s_{N+1},s_{N+2} \}= \overline{\Gamma_a\setminus\Gamma_a'}\cap \overline{\Gamma_a'}.
\end{align*}
The set $\{s_0,\ldots,s_N\}$ is for problem \eqref{eqn:problem_NR}, and the set $\{s_0,\ldots,s_{N+2}\}$ is for problem \eqref{eqn:problem_DR}. The solution $u$ to problem \eqref{eqn:problem_DR} has singularity at the vertices $s_{N+1}$ and $s_{N+2}$ due to the mix Dirichlet-Neumann boundary conditions. We denote by $\omega_j$ the angle of the boundary  at the transition point $s_j$. Since the domain $\Omega$ is convex, $\omega_j\in (0,\pi]$, and moreover $\omega_j=\pi$ if $s_j$ is located on an edge of $\Omega$, and $\omega_j< \pi$ if $s_j$ is a corner point. The following regularity decomposition holds; see e.g., \cite[Theorem 5.1.3.5]{Grisvard:1985} and \cite[Theorem 5.1]{Mghazli:1992}.
\begin{prop}\label{prop:sol_decomp}
Let $g\in H^\frac{1}{2}(\Gamma_a)$, and $q\in \mathcal{A}$ or $\mathcal{B}$. Then there exist constants $C_{j}$ such that  the solution $u $ to problem \eqref{eqn:problem_NR} satisfies
    \begin{equation*}
        u-\sum_{j=0}^{N} C_{j} u_{S_{j}}\in H^{2}(\Omega),
    \end{equation*}
    and the solution $u $ to problem \eqref{eqn:problem_DR} satisfies
    \begin{equation*}
        u-\sum_{j=0}^{N+2} C_{j} u_{S_{j}}\in H^{2}(\Omega),
    \end{equation*}
    where the singular functions $u_{S_{j}}$ are given by
    \begin{equation*}
        u_{S_{j}} = \left\{
            \begin{aligned}
                &O(|x-s_j|^{\pi/\omega_j}\log|x-s_j|), &&\mbox{ if }j=0,\dots,N,\\
                &O(|x-s_j|^{\pi/(2\omega_j)}\log|x-s_j|), &&\mbox{ if }j=N+1,N+2.\\
            \end{aligned}
         \right.
    \end{equation*}
\end{prop}

Note that the convexity of the domain $\Omega$ and Proposition \ref{prop:sol_decomp} imply that the solution $u$ to problem \eqref{eqn:problem_NR} has an almost $H^2(\Omega)$ regularity:
\begin{equation}\label{eqn:reg-0}
\|  u  \|_{H^{2-\epsilon}(\Omega)} \le c_\epsilon,\quad \forall \epsilon\in(0,1].
\end{equation}
For the exact data $f$ and the exact Robin coefficient $q^\dag$, the solution $u$ to problem  \eqref{eqn:problem_DR} coincides with that of problem \eqref{eqn:problem_NR} and hence it has the same regularity.

We use frequently the following interpolation inequality  \cite[Theorem 1.6.6]{Brenner:2008}.
\begin{lem}\label{lem:interp-ineq}
Let $\Omega$ be a Lispschitz domain. Then there holds $\| v \|_{L^2(\partial\Omega)} \le \| v \|_{L^2(\Omega)}^{\frac12} \| v
 \|_{H^1(\Omega)}^{\frac12}$ for all $v\in H^1(\Omega)$    
\end{lem}

\subsection{Finite element approximation}\label{subsec:FEM_approx}
The presence of the local solution singularity leads to sub-optimal convergence of the Galerkin FEM on uniform meshes.
The singularity analysis in Proposition \ref{prop:sol_decomp} motivates using locally refined FEM meshes in order to achieve optimal error estimates. Below we recall several standard results in the Galerkin FEM approximation. Let $\mathcal{T}_h$ be a triangulation of the domain $\Omega$ into triangles, denoted by $T$. Let $h_T$ be the diameter of $T$, $h=\max_{T\in \mathcal{T}_h} h_T$, and let $d_T=\min_{  j=N+1, N+2}\mathrm{dist}(s_j,T)$ denote the minimal distance from  the singular vertices $\{s_{N+1},s_{N+2}\} $ to the triangle $T$. We construct a mesh $\mathcal{T}_h$ graded near the vertices $\{s_j\}_{j=N+1,N+2} $ as (with  $h_*\sim h^{1/r}$)
\begin{equation}\label{eqn:mesh}
    \hbar\sim\left\{
        \begin{aligned}
            & d_T^{1-r} h, \quad \mbox{with }r\in(0,\min(1,\tfrac{\pi}{2\omega_j})), && \mbox{ for } h_*<d_T,\\
            & h_* , && \mbox{ for } h_*\ge d_T.
        \end{aligned}
    \right.
\end{equation}
Over the mesh $\mathcal{T}_h$, we define a continuous piecewise linear FEM space $V_h$ by
\begin{align*}
	V_h&=\{v_h\in H^1\II :v_h|_T \mbox{ is linear,} \quad \forall T\in \mathcal{T}_h\},
\end{align*}
and let $V_{h}^0=V_h\cap H_0^1(\Omega)$ and $V_h^{0,a}:=V_h\cap\{v_h=0\mbox{ on }\Gamma_a' \}$.
We define the  Ritz projection operator $R_h=R_h(q):H^1(\Omega)\rightarrow V_h$ by
\begin{equation*}
     (\nabla R_h v, \nabla \varphi_h)+(qR_hv, \varphi_h)_{L^2(\Gamma_i)}= (\nabla v,\nabla \varphi_h)+(q v, \varphi_h)_{L^2(\Gamma_i)},\quad \forall v\in H^1(\Omega),\, \varphi_h\in V_h.
\end{equation*}
The operator  $R_h$ satisfy the following error estimates for $s\in [1,2] $\cite[p. 32]{Thomee:2007}
\begin{align}
\label{eqn:error_Rh}
            \|v-R_h v\|_{L^2(\Omega)}+h\|\nabla(v-R_h v)\|_{L^2(\Omega)}\le ch^s\|v\|_{H^s(\Omega)}.
\end{align}
Let $\mathcal{I}_h$ be the Lagrange interpolation operator associated with the FEM space $V_h$. It satisfies the following error estimate for $s\in [1,2]$ and $1\le p\le \infty$ (with $sp>d$ if $p>1$ and $sp\ge d $ if $p=1$)\cite[Theorem 4.4.20]{Brenner:2008}:
\begin{equation}\label{eqn:error_Ih}
     \|v-\mathcal{I}_h v\|_{L^p(\Omega)}+h\|\nabla(v-\mathcal{I}_h v)\|_{L^p(\Omega)}\le ch^s\|v\|_{W^{s,p}(\Omega)}.
\end{equation}
We denote by $\mathcal{I}_h^\partial$ the interpolation operator restricted to the boundary $\partial\Omega$.

\subsection{Error analysis} \label{subsec:error_ell}

Now we develop a reconstruction algorithm for the elliptic inverse Robin problem and provide an error analysis of the scheme. The algorithm is based on the following Kohn-Vogelius functional that matches the energy in the domain \cite{KohnVogelius:1987}
\begin{equation}\label{eqn:cts_functional}
		\min_{q\in \mathcal{A}} J_{\al}(q)=\|\nabla u_N-\nabla u_D \|_{L^2(\Omega)}^2+\|  q^{\frac{1}{2}}(u_N-  u_D) \|_{L^2(\Gamma_i)}^2+\alpha\|  u_N-  u_D \|_{L^2(\Omega)}^2,
\end{equation}
where $\al>0$ is the penalty parameter, and $ u_N(q), u_D(q) \in H^1\II$ are weak solutions to problems \eqref{eqn:problem_NR} and \eqref{eqn:problem_DR}, respectively. The first two terms represent the standard Kohn-Vogelius functional for data fitting (similar to the weak formulation of problem \eqref{eqn:problem_NR}), and the third term is an additional control term. It controls the data-fitting in the $L^2(\Omega)$ norm, and plays a central role in the error analysis of the numerical scheme below. This  $L^2(\Omega)$ bound and the first term give an $H^1(\Omega)$ estimate and also yield an optimal bound on $\|u_N-u_D\|_{L^2(\Gamma_a')}$, cf. Lemma \ref{lem:apriori}. 
Note that when $q=q^\dag$, the function $u_N(q^\dag)$ is also a solution to problem \eqref{eqn:problem_DR}, and $u_D(q^\dag)$ coincides with $u_N(q^\dag)$ and the objective $J_\alpha(q^\dag)$ vanishes.

Now we can formulate the Galerkin FEM approximation of problem \eqref{eqn:cts_functional}:
\begin{equation}\label{eqn:dis_functional}
	\min_{q\in \mathcal{A}} J_{\al,h}(q)=\|\nabla u_{N,h}-\nabla u_{D,h}^\delta \|_{L^2(\Omega)}^2+\|  q^{\frac{1}{2}}(u_{N,h}-  u_{D,h}^\delta) \|_{L^2(\Gamma_i)}^2+\alpha\|  u_{N,h}-  u_{D,h}^\delta \|_{L^2(\Omega)}^2,
\end{equation}
where $ u_{N,h}=u_{N,h}(q) \in V_h$ solves
\begin{equation}\label{eqn:dis_constraint_NR}
	(\nabla u_{N,h},\nabla v_h)+(q u_{N,h} ,v_h)_{L^2(\Gamma_i)}=(g,v_h)_{L^2(\Gamma_a)},\quad \forall v_h\in V_h,
\end{equation}
and $  u_{D,h}^\delta=u_{D,h}^\delta(q)\in V_h$ satisfies {$u_{D,h}^\delta=  z_h^{\delta}\mbox{ on }\Gamma_a'$} and
\begin{equation}\label{eqn:dis_constraint_DR}
	(\nabla u_{D,h}^\delta,\nabla v_h)+(q u_{D,h}^\delta, v_h)_{L^2(\Gamma_i)}=(g, v_h)_{L^2(\Gamma_a\setminus\Gamma_a')},\quad \forall v_h\in V_h^{0,a}.
\end{equation}
The noisy data $z^\delta_h$ is used in problem \eqref{eqn:dis_constraint_DR} (for $u_{D,h}^\delta$).  Throughout, the noisy data $z_h^\delta \in C(\overline{\Gamma}_a')$ is assumed to satisfy a noise level $\delta$ in the following sense:
\begin{equation}\label{eqn:noise}
    \|z_h^{\delta}- f \|_{H^{s}(\Gamma_a')}\le c\delta^{1-2s/3},\quad \text{with } -\tfrac{1}{2}\le s \le \tfrac{1}{2}.
\end{equation}
\begin{remark}\label{rem:noise}
Condition \eqref{eqn:noise} assumes that the noise level changes with the norm $\|\cdot\|_{H^s(\Gamma_a')}$, which is reasonable in practice. Indeed, suppose that the observational data $z^\delta\in L^2(\Gamma_a')$ satisfies $\|z^{\delta}- u|_{\Gamma_a'} \|_{H^{-\frac{1}{2}}(\Gamma_a')}\le \delta^{\frac{4}{3}} $. Then let $z_h^\delta = P_{h,\Gamma_a'} z^\delta$, with $P_{h,\Gamma_a'}$ being the $L^2(\Gamma_a')$ orthogonal projection operator from $L^2(\Gamma_a')$ to $V_h|_{\Gamma_a'}$. The stability of $P_{h,\Gamma_a'}$ yields
\begin{equation}
\begin{aligned}
\| z_h^{\delta} -f\|_{H^{-\frac{1}{2}}(\Gamma_a')}
&\le \| z_h^{\delta} -P_{h,\Gamma_a'}f\|_{H^{-\frac{1}{2}}(\Gamma_a')} + \|P_{h,\Gamma_a'}f-f\|_{H^{-\frac{1}{2}}(\Gamma_a')} \\
&\leq c\| z^{\delta} -f\|_{H^{-\frac{1}{2}}(\Gamma_a')} + ch^2\| f \|_{H^\frac32(\Gamma_a')}\le c(\delta^{\frac{4}{3}}+h^2).
\end{aligned}
\end{equation}
where in the second inequality we use the estimate for $\phi\in H^\frac12(\Gamma_a')$
\begin{equation*} 
 |(P_{h,\Gamma_a'}f-f,\phi )_{\Gamma_a'}| = 
 | (P_{h,\Gamma_a'}f-f,\phi- P_{h,\Gamma_a'}\phi)_{\Gamma_a'} | \le c h^2 \| f \|_{H^\frac32(\Gamma_a')} \| \phi \|_{H^\frac12(\Gamma_a')} .
\end{equation*}
Meanwhile, by the inverse inequality \cite[(1.12)]{Thomee:2007} and the $L^2$ projection error \cite[p. 32]{Thomee:2007}, we derive
    \begin{align*}
         \| z_h^{\delta}- f \|_{H^{\frac{1}{2}}(\Gamma_a')}
         \le &  \|  P_{h,\Gamma_a'}(z^{\delta}-   f) \|_{H^{\frac{1}{2}}(\Gamma_a')}+ \| P_{h,\Gamma_a'} f- f \|_{H^{\frac{1}{2}}(\Gamma_a')}\\
         \le & ch^{-1}\|P_{h,\Gamma_a'} (z^{\delta}- f) \|_{H^{-\frac{1}{2}}(\Gamma_a')}+ch\|  f \|_{H^{ \frac{3}{2}}(\Gamma_a')}\\
         \le &ch^{-1}\| z^{\delta}- f \|_{H^{-\frac{1}{2}}(\Gamma_a')}+ch \le ch^{-1} \delta^{\frac{4}{3}}+ch.
    \end{align*}
By taking $h\sim\delta^{\frac{2}{3}}$, we arrive at the estimate $\|P_h z^{\delta}- f \|_{H^{\frac{1}{2}}(\Gamma_a')}\le c\delta^{\frac{2}{3}} $. The intermediate case follows by interpolation.
\end{remark}

Since the admissible sets $\mathcal{A}$ and $\mathcal{B}$ contain only piecewise constant functions, we do not need to discretize the Robin coefficient $q$ in the regularized formulation \eqref{eqn:dis_functional}.
In view of the finite-dimensionality of the parameter space, a standard compactness argument shows that the discrete problem \eqref{eqn:dis_functional}-\eqref{eqn:dis_constraint_DR} has at least one global minimizer $q^*$. Next we derive bounds on the error between the exact Robin coefficient $q^{\dag}$ and a global discrete minimizer  $q^*$.
We make the following assumption on the regularity of problem data.
\begin{assum}\label{assum:regularity}
The flux  $g\in H^{\frac{1}{2}}(\Gamma_a)$ with
  $ \|g\|_{H^{\frac{1}{2}}(\Gamma_a)}\le c_g$, the exact data $f=u(q^\dag)|_{\Gamma_a'} \in H^2(\Gamma_a')$, and the noisy data $z_h^\delta\in C(\overline{\Gamma}_a')$ satisfies \eqref{eqn:noise}.
\end{assum}
First we give several \textit{a priori} estimates on the data fitting term.
\begin{lem}\label{lem:error_L2}
Let Assumption \ref{assum:regularity} hold, and $q^\dag\in \mathcal{A}$ or $\mathcal{B}$. Let $u(q^{\dag})$, $u_{N,h}(q^\dag)$ and $u_{D,h}^\delta(  q^{\dag})$ be the solutions of problems \eqref{eqn:problem_NR}, \eqref{eqn:dis_constraint_NR} and \eqref{eqn:dis_constraint_DR}, respectively. Then  there hold
    \begin{align*}
      & \|  u_{D,h}^\delta( q^{\dag})- u(q^{\dag}) \|_{L^2(\Omega)}\le c(h^{2(1-\epsilon)} + \delta^{\frac{4}{3}(1-\epsilon)} ),\qquad \| \nabla u_{D,h}^\delta( q^{\dag})- \nabla u(q^{\dag}) \|_{L^2(\Omega)} \le  c(h^{1-\epsilon} + \delta^{\frac{2}{3}(1-\epsilon)} ),\\
     &  \|  u_{N,h}( q^{\dag})- u(q^{\dag}) \|_{L^2(\Omega)}+ h^{1-\epsilon}\| \nabla u_{N,h}( q^{\dag})- \nabla u(q^{\dag}) \|_{L^2(\Omega)} \le ch^{2(1-\epsilon)}.
    \end{align*}
\end{lem}
\begin{proof}
Let $u=u(q^\dag)$, $u_{N,h}=u_{N,h}(q^\dag)$ and $u_{D,h}^\delta=u_{D,h}^\delta(  q^{\dag})$.
We only show the error estimates for  $u_{D,h}^\delta$, and the proof of that for $u_{N,h}$ is analogous.  Let $u_{D,h}^0\in V_h$ with $u_{D,h}^0=\mathcal{I}_h^{\partial} f\mbox{ on }\Gamma_a'$ ($\mathcal{I}_h^{\partial} $ is the restriction of  $\mathcal{I}_h$ onto $\Gamma_a'$) solve
\begin{equation*}
	(\nabla u_{D,h}^0,\nabla v_h)+(q u_{D,h}^0,v_h)_{L^2(\Gamma_i)}=(g,v_h)_{L^2(\Gamma_a\setminus\Gamma_a')},\quad \forall v_h\in V_h^{0,a}.
\end{equation*}
 By the triangle inequality, we have
\begin{equation*}    \|\nabla(u_{D,h}^\delta- u)\|_{L^2(\Omega)}=\|\nabla(u_{D,h}^\delta- u_D)\|_{L^2(\Omega)}\le \|\nabla(u_{D,h}^\delta- u_{D,h}^0)\|_{L^2(\Omega)}+\|\nabla(u_{D,h}^0- u_D)\|_{L^2(\Omega)}.
\end{equation*}
The weak formulations of $u_{D,h}^0$ and $u_D$ with the test function $v_h=u_{D,h}^0-\mathcal{I}_h u_D$ and Galerkin orthogonality yield
\begin{align*}
    c\|\nabla(u_{D,h}^0- u_D)\|_{L^2(\Omega)}^2\le & (\nabla (u_{D,h}^0-u_D),\nabla (u_{D,h}^0-u_D))+(q^\dag (u_{D,h}^0-u_D), u_{D,h}^0-u_D )_{L^2(\Gamma_i)}\\
    = & (\nabla (u_{D,h}^0-u_D),\nabla (\mathcal{I}_h u_D-u_D))+(q^\dag (u_{D,h}^0-u_D), \mathcal{I}_h u_D-u_D )_{L^2(\Gamma_i)}\\
    \le & c \| u_{D,h}^0- u_D \|_{H^1(\Omega)} \| \mathcal{I}_h u_D- u_D \|_{H^1(\Omega)}.
\end{align*}
This, the interpolation error estimate \eqref{eqn:error_Ih} and the regularity estimate \eqref{eqn:reg-0} imply
\begin{equation*}
    \|\nabla(u_{D,h}^0- u_D)\|_{L^2(\Omega)}\le c\| \mathcal{I}_h u_D- u_D \|_{H^1(\Omega)}\le ch^{1-\epsilon}.
\end{equation*}
To bound the term $\|\nabla(u_{D,h}^\delta- u_{D,h}^0)\|_{L^2(\Omega)}$, let $G$ be the solution to
\begin{equation*}
    \left\{
        \begin{aligned}
            -\Delta G&=0, &&\mbox{ in }\Omega,\\
            \partial_n G+ q^\dag G&=0, &&\mbox{ on }\Gamma_i,\\
            \partial_nG&=0, &&\mbox{ on }\Gamma_a\setminus \Gamma_a',\\
             G&=  z_h^\delta- \mathcal{I}_h^\partial f, &&\mbox{ on }\Gamma_a'.
        \end{aligned}
    \right.
\end{equation*}
The weak formulations of $u_{D,h}^\delta$ and $ u_{D,h}^0$ and the Cauchy-Schwartz inequality yield
\begin{equation*}
    \|u_{D,h}^\delta- u_{D,h}^0 \|_{H^1(\Omega)}^2\le c\|u_{D,h}^\delta- u_{D,h}^0 \|_{H^1(\Omega)}  \|\mathcal{I}_h G \|_{H^1(\Omega)}.
\end{equation*}
The stability of the Lagrange interpolation $\mathcal{I}_h$ in $H^s(\Omega)$ with $2s>d$, Assumption \ref{assum:regularity} and condition \eqref{eqn:noise} imply
\begin{align}
    \|u_{D,h}^\delta- u_{D,h}^0 \|_{H^1(\Omega)}& \le c   \|\mathcal{I}_h G \|_{H^1(\Omega)}
    \le  c\|\mathcal{I}_h G\|_{H^{1+\epsilon}(\Omega)} \le c\| G \|_{H^{1+\epsilon}(\Omega)}
    \le c\|  z_h^\delta- \mathcal{I}_h^\partial f \|_{H^{\frac{1}{2}+\epsilon}(\Gamma_a')}\nonumber\\
    & \le c(\|   z_h^\delta-f  \|_{H^{\frac{1}{2}+\epsilon}(\Gamma_a')} + \|  \mathcal{I}_h^\partial f -f  \|_{H^{\frac{1}{2}+\epsilon}(\Gamma_a')})\le  c(\delta^{\frac{2}{3}(1-\epsilon)} + h^{\frac{3}{2}-\epsilon}).\label{eqn:udh}
\end{align}
This proves the bound on $\|\nabla(u_{D,h}^\delta- u)\|_{L^2(\Omega)}$. Next, we bound $\|u_{D,h}^\delta-u\|_{L^2(\Omega)}$.  By the triangle inequality,
\begin{equation*}
    \|u_{D,h}^\delta- u \|_{L^2(\Omega)}=\|u_{D,h}^\delta- u_D \|_{L^2(\Omega)}\le \|u_{D,h}^\delta- u_{D,h}^0\|_{L^2(\Omega)}+\|u_{D,h}^0- u_D \|_{L^2(\Omega)}.
\end{equation*}
By the triangle inequality,
\begin{equation*}
    \|u_{D,h}^\delta- u \|_{L^2(\Omega)}=\|u_{D,h}^\delta- u_D \|_{L^2(\Omega)}\le \|G\|_{L^2(\Omega)}+\|u_{D,h}^\delta- u_{D,h}^0-G\|_{L^2(\Omega)}+\|u_{D,h}^0- u_D \|_{L^2(\Omega)}.
\end{equation*}
The elliptic regularity theory \cite[Chapter 2, Theorem 6.6]{Lions:1972}, the approximation error \eqref{eqn:error_Ih} of the interpolation operator $\mathcal{I}_h^\partial$, Assumption \ref{assum:regularity} and condition \eqref{eqn:noise} imply
\begin{equation*}
    \|G\|_{L^2(\Omega)}\le c\|z_h^\delta-\mathcal{I}_h^\partial f\|_{H^{-\frac{1}{2}}(\Gamma_a')}\le c\|z_h^\delta-  f\|_{H^{-\frac{1}{2}}(\Gamma_a')}+c\| \mathcal{I}_h^\partial f-f \|_{H^{-\frac{1}{2}}(\Gamma_a')} \le c(\delta^{\frac{4}{3}}+h^2).
\end{equation*}
For the term $ \|u_{D,h}^\delta- u_{D,h}^0-G \|_{L^2(\Omega)}$, we employ a duality argument.  Let $\psi$ solve
    \begin{equation*}
        \left\{
            \begin{aligned}
                 -\Delta \psi & = u_{D,h}^\delta- u_{D,h}^0-G, && \mbox{in }\Om, \\
                \partial_n\psi+q^\dag \psi & =0, && \mbox{on } \Gamma_i,\\
                \partial_n \psi  & =0, && \mbox{on } \Gamma_a\setminus\Gamma_a',\\
                \psi & =0, && \mbox{on } \Gamma_a'.\\
            \end{aligned}
        \right.
    \end{equation*}
Then integration by parts, the weak formulations of $u_{D,h}^\delta$, $u_{D,h}^0 $ and $G$, and the Galerkin orthogonality yield
\begin{align*}
&\|u_{D,h}^\delta- u_{D,h}^0-G\|_{L^2(\Omega)}^2= - (u_{D,h}^\delta- u_{D,h}^0-G,\Delta \psi)\\
=&(\nabla(u_{D,h}^\delta  - u_{D,h}^0-G),\nabla\psi)
+(q^\dag(u_{D,h}^\delta- u_{D,h}^0-G),\psi)_{L^2(\Gamma_i)} -(u_{D,h}^\delta- u_{D,h}^0-G,\partial_n\psi)_{L^2(\Gamma_a')}\\
 =&(\nabla(u_{D,h}^\delta- u_{D,h}^0-G),\nabla(\psi-\mathcal{I}_h \psi))+(q^\dag(u_{D,h}^\delta- u_{D,h}^0-G),\psi-\mathcal{I}_h\psi)_{L^2(\Gamma_i)}\\
\le & c\big(\|u_{D,h}^\delta- u_{D,h}^0 \|_{H^1(\Omega)}+\|G\|_{H^1(\Omega)}  \big)\| \psi-\mathcal{I}_h\psi \|_{H^1(\Omega)}.
\end{align*}
By Proposition \ref{prop:sol_decomp}, $\psi$ has weak singularities due to the mixed Dirichlet-Neumann boundary conditions. The construction of the locally refined mesh $\mathcal{T}_h$ ensures the following optimal error estimate \cite[Lemma 19.10]{Thomee:2007}
    \begin{equation*}
        \| \psi-\mathcal{I}_h\psi \|_{H^1(\Omega)}\le ch^{1-\epsilon}\| u_{D,h}^\delta -u_{D,h}^0-G\|_{L^2(\Omega)}.
    \end{equation*}
By the bound \eqref{eqn:udh}, we get
    \begin{equation*}
        \|u_{D,h}^\delta -u_{D,h}^0-G\|_{L^2(\Omega)}
        \le c(h^{1-\epsilon}\delta^{\frac{2}{3}(1-\epsilon)}+h^{\frac{5}{2}-2\epsilon}).
\end{equation*}
Similarly, let $\phi$ solve
    \begin{equation*}
        \left\{
            \begin{aligned}
                 -\Delta \phi & =u_{D,h}^0-u_D, && \mbox{in }\Om, \\
                \partial_n\phi+q^\dag \phi & =0, && \mbox{on } \Gamma_i,\\
                \partial_n\phi  & =0, && \mbox{on } \Gamma_a\setminus\Gamma_a',\\
                \phi & =0, && \mbox{on } \Gamma_a'.\\
            \end{aligned}
        \right.
    \end{equation*}
Then integration by parts, the weak formulations of $u_D$ and $u_{D,h}^0$ and the Galerkin orthogonality yield
    \begin{align*}
         &\|u_{D,h}^0 -u_D\|_{L^2(\Omega)}^2= (\nabla(u_{D,h}^0-u_D),\nabla\phi)-(u_{D,h}^0-u_D,\partial_n\phi)_{L^2(\partial\Omega)} \\
        \le & c \|u_{D,h}^0-u_D \|_{H^1(\Omega)}\| \phi-\mathcal{I}_h\phi \|_{H^1(\Omega)} +c\|\mathcal{I}_h^\partial f-f\|_{L^2(\Gamma_a')}\|\partial_n\phi\|_{L^2(\Gamma_a')}.
    \end{align*}
By the construction of the locally refined mesh $\mathcal{T}_h$, the elliptic regularity \eqref{eqn:reg-0} and the regularity $f\in H^2(\Gamma_a')$ (cf. Assumption \ref{assum:regularity}), we deduce 
$$\|u_{D,h}^0 -u_D\|_{L^2(\Omega)}\le ch^{2(1-\epsilon)}.$$
Then the triangular inequality yields the desired result.
\end{proof}
Now we can state the following \textit{a priori} estimate on the objective $J_{\alpha,h}(q^*)$.
\begin{lem}\label{lem:apriori}
Let   Assumption \ref{assum:regularity} hold, and $q^*$ be a minimizer of problem \eqref{eqn:dis_functional}-\eqref{eqn:dis_constraint_DR}. Then  there holds
    \begin{equation*}
       J_{\alpha,h}(q^*)\le c(\eta +\alpha \eta^2),\quad \mbox{with }\eta = h^{2(1-\epsilon)}+\delta^{\frac{4}{3}(1-\epsilon)}.
    \end{equation*}
\end{lem}
\begin{proof}
Since $q^*$ minimizes $J_{\alpha,h}$, we have $J_{\alpha,h}(q^*)\le J_{\alpha,h}( q^{\dag})$. It suffices to bound $J_{\alpha,h}(q^\dag)$. Let $u^\dag=u(q^\dag)$ be the solution of problem \eqref{eqn:problem_NR} with the Robin coefficient $q^\dag$. By Lemma \ref{lem:error_L2}, we have
    \begin{align*}
        \|\nabla u_{N,h}(q^\dag)-\nabla u_{D,h}^\delta(q^\dag) \|_{L^2(\Omega)}&\le  \|\nabla u_{N,h}(q^\dag)-\nabla u^\dag \|_{L^2(\Omega)}+ \|\nabla u^\dag-\nabla u_{D,h}^\delta(q^\dag) \|_{L^2(\Omega)}\le c(h^{1-\epsilon}+\delta^{\frac{2}{3}(1-\epsilon)}),\\
        \| u_{N,h}(q^\dag)- u_{D,h}^\delta(q^\dag) \|_{L^2(\Omega)}&\le  \| u_{N,h}(q^\dag)- u^\dag \|_{L^2(\Omega)}+ \| u^\dag- u_{D,h}^\delta(q^\dag) \|_{L^2(\Omega)}\le c(h^{2(1-\epsilon)}+\delta^{\frac{4}{3}(1-\epsilon)}).
    \end{align*}
    Furthermore, the box constraint, Lemma \ref{lem:interp-ineq}, and Young's inequality imply
    \begin{align*}
        &\|  q^{\frac{1}{2}}(u_{N,h}(q^\dag)-  u_{D,h}^\delta(q^\dag)) \|_{L^2(\Gamma_i)}
        \le  c\|   u_{N,h}(q^\dag)-  u_{D,h}^\delta(q^\dag) \|_{L^2(\Gamma_i)}\\
        \le & c\|   u_{N,h}(q^\dag)-  u_{D,h}^\delta(q^\dag) \|_{L^2(\Omega)}^{\frac{1}{2}}\|   u_{N,h}(q^\dag)-  u_{D,h}^\delta(q^\dag) \|_{H^1(\Omega)}^{\frac{1}{2}}
        \le c(h^{\frac{3}{2}(1-\epsilon)}+\delta^{1-\epsilon}).
    \end{align*}
This completes the proof of the lemma.
\end{proof}

Now we can state an error estimate on the approximation $q^*$.
\begin{thm}\label{thm:error_estimate}
Let   Assumption \ref{assum:regularity} hold, and $q^*$ be a minimizer of problem \eqref{eqn:dis_functional}-\eqref{eqn:dis_constraint_DR}. Let  $\eta=h^{2(1-\epsilon)} + \delta^{\frac{4}{3}(1-\epsilon)}$.
\begin{itemize}
    \item[{\rm(i)}] If $q^\dag\in \mathcal{A}$, then there exists $c>0$ independent of $h$, $\delta$ and $\alpha$ such that
$$\| q^*-q^{\dag} \|_{L^\infty(\Gamma_i)} \le  c((\alpha^{-\frac14} +\alpha^{-\frac12})(\eta^\frac12+\alpha^\frac12\eta)+\eta^\frac{3}{4}).$$
\item[{\rm(ii)}] If $q^\dag\in \mathcal{B}$, then there exists $c>0$ independent of $h,\delta,\alpha$, $\theta=\frac{1-2\epsilon}{3-2\epsilon}$ with $ \epsilon\in(0,\frac{1}{2})$, and $\kappa=\kappa(c_g, \underline{c}_q, \Bar{c}_q, c_0, \Omega) \in(0,1)$ such that 
$$\| q^*-q^{\dag} \|_{L^\infty(\Gamma_i)} \le c((\alpha^{-\frac14} +\alpha^{-\frac12})(\eta^\frac12+\alpha^\frac12\eta)+\eta^\frac34)^{\kappa\theta}.$$
\end{itemize}
\end{thm}
\begin{proof}
In case (i), since the constant portions are fixed, the following stability result \cite[Theorem 2.4]{Sincich:2007} holds 
    \begin{equation*}
       \| q^*-q^{\dag} \|_{L^\infty(\Gamma_i)} \le c\| u(q^*)-u(q^{\dag}) \|_{L^2(\Gamma_a')}.
    \end{equation*}
In case (ii), both $q^{\dag}$ and $q^*$ are piecewise constant with at most $N$ discontinuous points and each piece has a measure larger than $c_0$. By \cite[Theorem 4.1]{Hu:2015} and the interpolation inequality, for $\theta=\frac{1-2\epsilon}{3-2\epsilon}$ and $\kappa\equiv \kappa(c_g,  \underline{c}_q,\Bar{c}_q,c_0,\Omega) \in(0,1)$, we have
    \begin{equation*}
       \| q^*-q^{\dag} \|_{L^\infty(\Gamma_i)} \le  c\| u(q^*)-u(q^{\dag}) \|_{H^1(\Gamma_a')}^{\kappa} \le c\| u(q^*)-u(q^{\dag}) \|_{L^2(\Gamma_a')}^{\kappa\theta}\| u(q^*)-u(q^{\dag}) \|_{H^{\frac{3}{2}-\epsilon}(\Gamma_a')}^{\kappa (1-\theta)} ,
    \end{equation*}
 Then the trace theorem and the regularity estimate \eqref{eqn:reg-0} imply
    \begin{equation*}
        \| u(q^*)-u(q^{\dag}) \|_{H^{\frac{3}{2}-\epsilon}(\Gamma_a')} \le c\| u(q^*)-u(q^{\dag}) \|_{H^{2-\epsilon}(\Omega)} \le c(\epsilon, c_g,\underline{c}_q,\Bar{c}_q,c_0).
    \end{equation*} 
Thus, in either case, it suffices to estimate $\|u(q^*)-u(q^{\dag}) \|_{L^2(\Gamma_a')} $. By the triangle inequality, we obtain
    \begin{align*}
        &\|u(q^*)-u(q^{\dag}) \|_{L^2(\Gamma_a')}
        \le \|u(q^*)-u_{N,h}(q^*) \|_{L^2(\Gamma_a')} + \|u_{N,h}(q^*)-  z_h^\delta \|_{L^2(\Gamma_a')} +\| z_h^\delta-u(q^{\dag}) \|_{L^2(\Gamma_a')}.
    \end{align*}
By repeating the argument for Lemma \ref{lem:error_L2} and the regularity of $u(q^*)$, we get
    \begin{align*}
        \|u(q^*)-u_{N,h}(q^*) \|_{L^2(\Gamma_a')}&\le c \|u(q^*)-u_{N,h}(q^*) \|_{L^2(\Omega)}^{\frac{1}{2}}\|u(q^*)-u_{N,h}(q^*) \|_{H^1(\Omega)}^{\frac{1}{2}}\\
        &\le ch^{\frac{3}{2}(1-\epsilon)}\| u(q^*) \|_{H^{2-\epsilon}(\Omega)}\le c(\epsilon, c_g,\underline{c}_q,\Bar{c}_q,c_0)h^{\frac{3}{2}(1-\epsilon)}.
       \end{align*}
By Lemma \ref{lem:interp-ineq}, we derive
    \begin{align*}
        &\|u_{N,h}(q^*)-  z_h^\delta \|_{L^2(\Gamma_a')}
        =\|u_{N,h}(q^*)-u_{D,h}^\delta(q^*)  \|_{L^2(\Gamma_a')}\\
        \le & c\|u_{N,h}(q^*)-u_{D,h}^\delta(q^*)  \|_{L^2(\Omega)}^{\frac{1}{2}} \|u_{N,h}(q^*)-u_{D,h}^\delta(q^*)  \|_{H^1(\Omega)}^{\frac{1}{2}}\\
        \le & c(\alpha^{-1}J_{\alpha,h}(q^*))^{\frac{1}{4}} (J_{\alpha,h}(q^*)+\alpha^{-1}J_{\alpha,h}(q^*))^{\frac{1}{4}}.
    \end{align*}
Now Lemma \ref{lem:apriori} implies $J_{\alpha,h}(q^*)^\frac{1}{2}\leq \eta^\frac12 + \alpha^\frac12\eta$, we deduce
\begin{align*}
    & (\alpha^{-1}J_{\alpha,h}(q^*))^{\frac{1}{4}} (J_{\alpha,h}(q^*)+\alpha^{-1}J_{\alpha,h}(q^*))^{\frac{1}{4}} \leq (\alpha^{-\frac14} +\alpha^{-\frac12})J_{\alpha,h}(q^*)^{\frac12}
      \le  (\alpha^{-\frac14} +\alpha^{-\frac12})(\eta^\frac12+\alpha^\frac12\eta).
\end{align*}
The definition \eqref{eqn:noise} of the noise level implies
    \begin{align*}
        \|  z_h^\delta-u(q^{\dag}) \|_{L^2(\Gamma_a')}= \|  z_h^\delta-f \|_{L^2(\Gamma_a')}\le c\delta.
       \end{align*}
Combining the preceding estimates with the stability estimates yields the assertion in cases (i) and (ii).
\end{proof}
\begin{remark}\label{rem:conv_rate_ell}
Theorem \ref{thm:error_estimate} provides a guideline for the \textit{a priori} choice of the algorithmic parameters $h$ and $\alpha$, in relation with $\delta$. The choice $h\sim \delta^{\frac{2}{3}}$ and $\alpha\sim \delta^{-\frac{4}{3}}$ yields a convergence rate $\| q_h^*-q^{\dag} \|_{L^\infty(\Gamma_i)} \le c\delta^{1-\epsilon} $. The convergence rate is nearly optimal when compared with the Lipschitz stability result \cite[Theorem 2.4]{Sincich:2007}. In practice, the choice of $\alpha$ greatly affects the reconstruction accuracy: Too large $\alpha$ leads to a predominant $L^2(\Omega)$-norm matching term, which neglects the energy information in the functional \eqref{eqn:cts_functional}. In practice, we choose parameter $\alpha\sim \delta^{-\frac{4}{3}}$ such that at the initial state, the magnitude of  error $ \alpha\|u_{N,h}-u_{D,h}^\delta \|_{L^2(\Omega)}^2$ is comparable with $ \|\nabla u_{N,h}-\nabla u_{D,h}^\delta \|_{L^2(\Omega)}^2+  \| q^{\frac{1}{2}}(u_{N,h}- u_{D,h}^\delta) \|_{L^2(\Omega)}^2 $.
\end{remark}

\section{Inverse Robin problem in the parabolic case}\label{sec:parabolic}

In this section, we extend the reconstruction algorithm in Section \ref{sec:elliptic} to the parabolic case:
\begin{equation}\label{eqn:para_robin}
    \left\{
        \begin{aligned}
            \partial_t u -\Delta u & =0, && \mbox{in }\Om\times (0,T], \\
            \partial_n u+qu & =0, && \mbox{on } \Gamma_i\times (0,T],\\
            \partial_n u & =g, && \mbox{on } \Gamma_a\times (0,T],\\
            u(0) & =u_0, && \mbox{in } \Omega.
        \end{aligned}
    \right.
\end{equation}
The inverse Robin problem is to recover the piecewise constant Robin coefficient $q(x)$ from the terminal boundary observation $u$ on $\Gamma_a'\times\{T\}$.
Below for a bivariate function $f(x,t)$, we often write $f(t)=f(\cdot,t)$. In the analysis, we impose the following condition on the problem data.
\begin{assum}\label{assum:regularity_para}
The problem data $g,u_0$ and $f$ and the noisy measurement $z^\delta$ satisfy following conditions.
\begin{itemize}
    \item[{\rm(i)}]  The time-independent flux $g$ satisfies $g\in H^{\frac{1}{2}}(\Gamma_a)$, the initial data $u_0\in H^{2-\epsilon}(\Omega)$, $\forall \epsilon\in(0,1]$ satisfying the compatibility condition $\partial_n u_0 +q u_0=0$ on $\Gamma_i$ and  $\partial_n u_0  =g$ on $\Gamma_a$  and there exist $c_0,c_g>0$ such that
    \begin{equation*}
        \|u_0\|_{H^{2-\epsilon} (\Omega)}\le c_\epsilon \quad\mbox{and}\quad \|g\|_{H^{\frac{1}{2}}(\Gamma_a)}\le c_g.
    \end{equation*}
    \item[{\rm(ii)}] The exact data $f=u(q^\dag)|_{\Gamma_a'\times\{T\} }\in H^2(\Gamma_a')$, and the noisy data $z_h^\delta\in C(\overline{\Gamma}_a')$ satisfies
    \begin{equation}\label{eqn:para_noise}
        \|z_h^{\delta}- f \|_{H^{s}(\Gamma_a')}\le C\delta^{1-2s/3},\quad \text{with } -\tfrac{1}{2}\le s \le \tfrac{1}{2}.
    \end{equation}
\end{itemize}
\end{assum}

Under Assumption \ref{assum:regularity_para}(i), problem \eqref{eqn:para_robin} has a unique solution $u\in C([0,T];H^1(\Omega))$. Moreover, $u(t)$ satisfies the regularity decomposition in  Proposition \ref{prop:sol_decomp}, for all $t\in [0,T]$. For the regularized reconstruction, we employ the following Kohn-Vogelius type functional:
\begin{equation}\label{eqn:para_cts_functional}
\min_{q\in \mathcal{A}} J_{\al}(q)= \|\nabla u_N(T)-\nabla u_D  \|_{L^2(\Omega)}^2+\|  q^{\frac{1}{2}}(u_N(T)-  u_D ) \|_{L^2(\Gamma_i)}^2+\alpha\|  u_N(T)-  u_D  \|_{L^2(\Omega)}^2,
\end{equation}
where $\al>0$ is the penalty parameter, $ u_N  \in H^1\II$ is the weak solution of problem \eqref{eqn:para_robin} and $ u_D  \in H^1\II$ is the weak solution of the elliptic problem
\begin{equation}\label{eqn:para_DR}
    \left\{
        \begin{aligned}
            -\Delta u_D & =-\partial_t u_N(T), && \mbox{in }\Om, \\
            \partial_nu_D+qu_D & =0, && \mbox{on } \Gamma_i,\\
            \partial_n u_D & =g, && \mbox{on } \Gamma_a\setminus\Gamma_a',\\
            u_D & =f, && \mbox{on } \Gamma_a',
        \end{aligned}
    \right.
\end{equation}
with $f=u(q^\dag)|_{\Gamma_a'\times\{T\}}$ being the exact data. Note that the formulation of the functional $J_\alpha(q)$ involves an elliptic and a parabolic problem. This choice is motivated by the availability of the data at the terminal time $T$ only.

\subsection{Stability estimate}\label{subsec:stab_para}
Now we derive new stability estimates for the parabolic inverse Robin problem. Let $A=A(q)$ be the $L^2(\Omega)$ realization of the  operator $-\Delta$ with its domain $D(A)\coloneqq \{v\in L^2(\Omega) \,:\, -\Delta v\in L^2(\Omega),\, \partial_n v+qv=0 \mbox{ on }\Gamma_i,\, \partial_n v =0 \mbox{ on }\Gamma_a \}$. The spectrum of the operator $A$ consists of positive eigenvalues $\{\lambda_j\}_{j=1}^{\infty}$ (ordered nondecreasingly), and the eigenfunctions $\{\phi_j\}_{j=1}^{\infty}$ can be chosen to form an orthonormal basis $L^2(\Omega)$. For any $s\geq0$, we define the fractional power $A^s$ by
$A^s v=\sum_{j=1}^{\infty}\lambda_j^s (v,\phi_j)\phi_j$, with its domain $D(A^s) = \{v\in L^2(\Omega):\, A^s v\in L^2(\Omega)\}$.
Then we define the solution operator $E(t)$ by
\begin{equation}\label{eqn:sol_op}
	E(t)=\frac{1}{2\pi \mathrm{i}}\int_{\Gamma_{\theta,\sigma }}e^{zt}(z+A)^{-1}\d z,
\end{equation}
where the contour $\Gamma_{\theta,\sigma}=\{z\in \mathbb{C}:|z|=\sigma,|\arg(z)|\le \theta\}\cup\{z\in \mathbb{C}:z=\rho e^{i\theta},\rho\ge \sigma\}$,
with fixed $\sigma\in (0,\infty)$ and $\theta\in(\pi/2,\pi)$, is oriented with an increasing imaginary part. The solution $u$ to problem \eqref{eqn:para_robin} can be represented by
\begin{equation*}
    u(t)=E(t)(u_0-\mathcal{N}g)+\mathcal{N}g,
\end{equation*}
where the Neumann lifting $w:=\mathcal{N}(q)g\in H^1(\Omega)$ solves the following elliptic problem
\begin{equation}\label{eqn:op_N}
   \left\{\begin{aligned} -\Delta w  & =0, &&\mbox{ in }\Omega,\\
    \partial_nw+qw  &= 0, &&  \mbox{ on } \Gamma_i,\\
    \partial_nw  &=g , &&\mbox{ on } \Gamma_a.
   \end{aligned}\right.
\end{equation}

Now we define an auxiliary function $\overline{u} = \mathcal{N}(q)g$, which solves the stationary problem \eqref{eqn:op_N}. It will be used in establishing the stability analysis. The following lemma bounds the difference between the solutions of   the parabolic  and  stationary problems. Throughout, let $\lambda(q)>0$ be the smallest eigenvalue of the operator $A=A(q)$.
\begin{lem}\label{lem:para-stationary}
Let  $\Bar{u}=\mathcal{N}g$, and  $u(t)$ be the solution of problem \eqref{eqn:para_robin}. Then there holds
\begin{equation*}
  \|u(t)-\Bar{u} \|_{H^1(\Omega)}\le ce^{-\frac{\lambda(q)}{2}t}\|u_0-\Bar{u} \|_{H^1(\Omega)}.
\end{equation*}
\end{lem}
\begin{proof}
For any $p,r\geq 0$, direct computation yields
\begin{align*}
        \|E(t) v\|_{D(A^r)  }^2\le \sup_{\lambda_j} \left(\lambda_j^{{2(r-p)}}  e^{-\lambda_j t} \right)e^{-\lambda(q) t} \sum_{j=1}^{\infty}\lambda_j^{2p}  (v,\phi_j)^2
        \le ct^{-2(r-p)} e^{-\lambda(q) t} \|   v\|_{D(A^p)  }^2,
\end{align*}
in view of the elementary inequality $\sup_{\lambda>0}\lambda^se^{-\lambda t} \leq ct^{-s}$ for $s>0$. Consequently,
\begin{align}
  \|E(t) v\|_{D(A^r)  } \leq ct^{p-r}e^{-\frac{\lambda(q)}{2}}\|v\|_{D(A^p)}.\label{eqn:E-smoothing}
\end{align}
This estimate and the representation $u(t)-\Bar{u}=E(t)(u_0-\Bar{u})$ directly imply
    \begin{equation*}
         \|u(t)-\Bar{u} \|_{H^1(\Omega)}\le  \|E(t) \|_{H^1(\Omega)\rightarrow H^1(\Omega)} \| u_0-\Bar{u} \|_{H^1(\Omega)}\le   ce^{-\frac{\lambda(q)}{2}t}\|u_0-\Bar{u} \|_{H^1(\Omega)}.
    \end{equation*}
This completes the proof of the lemma.
\end{proof}

Now we can state two new stability estimates for the parabolic inverse Robin problem.
\begin{thm}\label{thm:stab_para}
Let Assumption \ref{assum:regularity_para} hold. Let $f_i=u_i|_{\Gamma_a'\times\{T\}}$, $i=1,2$ be the exact data associated with the Robin coefficient $q_i$. Then the following stability estimates hold.
\begin{itemize}
    \item[{\rm(i)}] If $q_i\in \mathcal{A}$, $i=1,2$, then there exist $T_0$ such that for all $T>T_0$, there holds
    \begin{equation*}
        \|q_1-q_2\|_{L^\infty(\Gamma_i)}\le c\|f_1-f_2\|_{L^2(\Gamma_a')}.
    \end{equation*}
   \item[{\rm(ii)}]  If $q_i\in \mathcal{B}$, $i=1,2$, then there exist $r\in (\frac{1}{4},\frac{1}{2})$, $\theta=\frac{1-2\epsilon}{3-2\epsilon}$, and $\kappa=\kappa(c_g,\underline{c}_q,\overline{c}_q,c_0,\Omega)\in (0,1)$, such that
    \begin{equation*}
        \|q_1-q_2\|_{L^\infty(\Gamma_i)}\le c(\|f_1-f_2\|_{L^2(\Gamma_a')}+ T^{ \frac{1}{2}-r }e^{-\min(\frac{\lambda(q_1)}{2},\frac{\lambda(q_2)}{2})T} \|q_1-q_2\|_{L^\infty(\Gamma_i)})^{\kappa\theta} .
    \end{equation*}
\end{itemize}
\end{thm}
\begin{proof}
    Let $\overline{u}_i=\mathcal{N}(q_i)g$, $i=1,2$.
    The elliptic stability results in \cite[Theorem 2.4]{Sincich:2007} and \cite[Theorem 4.1]{Hu:2015} and interpolation inequality yield for $\theta=\frac{1-2\epsilon}{3-2\epsilon}$,
    \begin{equation*}
        \|q_1-q_2\|_{L^\infty(\Gamma_i)}\le \left\{
        \begin{aligned}
          c\|\Bar{u}_1-\Bar{u}_2\|_{L^2(\Gamma_a')}, &\quad q_i\in \mathcal{A},\\
          c\|\Bar{u}_1-\Bar{u}_2\|_{L^2(\Gamma_a')}^{\kappa\theta}, &\quad q_i\in \mathcal{B}.
        \end{aligned}\right.
\end{equation*}
By the triangle inequality, 
$$\|\Bar{u}_1-\Bar{u}_2\|_{L^2(\Gamma_a')}\le \|f_1-f_2\|_{L^2(\Gamma_a')}+  \|f_1-f_2-\Bar{u}_1+\Bar{u}_2\|_{L^2(\Gamma_a')}.$$ 
Let $w=u_1-u_2-\Bar{u}_1+\Bar{u}_2$. Then with $\psi=(q_2-q_1)(u_2-\Bar{u}_2) $, $w$ satisfies
    \begin{equation*}
    \left\{
        \begin{aligned}
            \partial_t w -\Delta w & =0, && \mbox{in }\Om\times (0,T], \\
            \partial_n w+q_1w & =\psi, && \mbox{on } \Gamma_i\times (0,T],\\
            \partial_nw & =0, && \mbox{on } \Gamma_a\times (0,T],\\
            w(0) & =\Bar{u}_1-\Bar{u}_2, && \mbox{in } \Omega.
        \end{aligned}
    \right.
    \end{equation*}
By integration by parts, the following solution representation holds
\begin{align}
w(t)=&\mathcal{R}\psi(t)+E(t)(\Bar{u}_1-\Bar{u}_2-\mathcal{R}\psi(0) )-\int_0^t E(t-s)\partial_t \mathcal{R}\psi(s) \d s\nonumber\\
=& E(t)(\Bar{u}_1-\Bar{u}_2  )+\int_0^t E'(t-s)  \mathcal{R}\psi(s) \d s,\label{eqn:sol-rep-w}
\end{align}
    where the function $\mathcal{R}\psi$ solves the elliptic problem
    \begin{equation}\label{eqn:Robin_sol}
    \left\{
        \begin{aligned}
              -\Delta \mathcal{R}\psi & =0, && \mbox{in }\Om , \\
            \partial_n\mathcal{R}\psi+q_1\mathcal{R}\psi & = \psi, && \mbox{on } \Gamma_i ,\\
            \partial_n\mathcal{R}\psi & =0, && \mbox{on } \Gamma_a.
        \end{aligned}
    \right.
    \end{equation}
Note that by Lemma \ref{lem:para-stationary}, we have the following elliptic regularity estimate
\begin{align}
    \|\mathcal{R}\psi(t)\|_{H^1(\Omega)} &\leq c\|\psi(t)\|_{L^2(\Gamma_i)} \leq c\|q_1-q_2\|_{L^\infty(\Gamma_i)}\|u_2(t)-\overline{u}_2\|_{L^2(\Gamma_i)}\nonumber\\
     & \leq ce^{-\frac{\lambda(q_2)}{2}t}\|q_1-q_2\|_{L^\infty(\Gamma_i)}.\label{eqn:est-R-psi}
\end{align}
Then by the representation \eqref{eqn:sol-rep-w},
    \begin{equation*}
        \|w(t)\|_{D(A^{r})}\le \|E(t)(\Bar{u}_1-\Bar{u}_2  )\|_{D(A^{r})}+\int_0^t \|E'(t-s)\|_{D(A^{\frac{1}{2} })\rightarrow D(A^{r})}  \|\mathcal{R}\psi(s)\|_{D(A^{\frac{1}{2} })} \d s:=\mathrm{I}+\mathrm{II},
    \end{equation*}
    with the exponent $r\in (\frac{1}{4},\frac{1}{2}) $.
    By the smoothing property \eqref{eqn:E-smoothing} and the elliptic regularity 
    \begin{equation*} 
    \|\overline{u}_1-\overline{u}_2\|_{H^1(\Omega)}\leq \|(q_2-q_1)\overline{u}_2\|_{L^2(\Gamma_i)}\leq c\|q_1-q_2\|_{L^\infty(\Gamma_i)},
    \end{equation*}
    we have
    \begin{align*}
        {\rm I} \le&  \|E(t)\|_{D(A^{\frac{1}{2}})\rightarrow D(A^{r})}\|\Bar{u}_1-\Bar{u}_2 \|_{D(A^{\frac{1}{2}})}\le ct^{\frac{1}{2}-r}e^{-\frac{\lambda(q_1)}{2}t}\|q_1-q_2\|_{L^\infty(\Gamma_i)}.
    \end{align*}
Meanwhile, the identity $-AE(t)=E'(t)$ and the estimate \eqref{eqn:est-R-psi} yield
\begin{align*}
{\rm II}\le& c \int_0^t (t-s)^{ -\frac{1}{2}-r} e^{-\frac{\lambda(q_1)}{2}(t-s)}  \|\mathcal{R}\psi(s)\|_{H^1(\Omega)} \d s\\
\le & c \int_0^t  (t-s)^{ -\frac{1}{2}-r} e^{-\frac{\lambda(q_1)}{2}(t-s)} e^{-\frac{\lambda(q_2)}{2}s} \d s \|q_1-q_2\|_{L^\infty(\Gamma_i)}
    \end{align*}
Direct computation gives the following elementary inequality
    \begin{equation*}
        \int_0^t  (t-s)^{ -\frac{1}{2}-r} e^{-\frac{\lambda(q_1)}{2}(t-s)} e^{-\frac{\lambda(q_2)}{2}s} \d s\le ce^{-\frac{\lambda(q_1)}{2} t} t^{\frac{1}{2}-r }\max(1,e^{(\frac{\lambda(q_1)}{2}-\frac{\lambda(q_2)}{2})t})\le ct^{ \frac{1}{2}-r }e^{-\min(\frac{\lambda(q_1)}{2},\frac{\lambda(q_2)}{2}) t}.
    \end{equation*}
Consequently,
    \begin{equation*}
        {\rm II} \le ct^{ \frac{1}{2}-r }e^{-\min(\frac{\lambda(q_1)}{2},\frac{\lambda(q_2)}{2}) t}\|q_1-q_2\|_{L^\infty(\Gamma_i)}.
    \end{equation*}
The estimates on $\mathrm{I}$ and $ \mathrm{II}$ and the trace theorem yield for $r\in (\frac{1}{4},\frac{1}{2})$,
\begin{equation*}
    \|w(t)\|_{L^2(\Gamma_a')}\le \|w(t)\|_{D(A^{r}) }\le ct^{ \frac{1}{2}-r }e^{-\min(\frac{\lambda(q_1)}{2},\frac{\lambda(q_2)}{2})t}\|q_1-q_2\|_{L^\infty(\Gamma_i)}.
\end{equation*}
For any $q_1,q_2\in \mathcal{A}$, by taking $t=T$ sufficiently large such that $cT^{ \frac{1}{2}-r }e^{-\min(\frac{\lambda(q_1)}{2},\frac{\lambda(q_2)}{2})T}<\frac{1}{2} $, we get
    \begin{equation*}
         \|q_1-q_2\|_{L^\infty(\Gamma_i)}\le c\left(\|f_1-f_2\|_{L^2(\Gamma_a')}+  \|w(T)\|_{L^2(\Gamma_a')}\right)\le c\|f_1-f_2\|_{L^2(\Gamma_a')}+\tfrac{1}{2}\|q_1-q_2\|_{L^\infty(\Gamma_i)}.
    \end{equation*}
This shows the assertion in case (i). For $q_1,q_2\in \mathcal{B}$, taking $t=T$ in the preceding estimate gives the desired assertion in case (ii).
\end{proof}

\subsection{Error analysis}\label{subsec:error_para}
Now we develop a discrete reconstruction algorithm for the parabolic inverse Robin problem and provide an error analysis of the scheme. We first present a fully discrete scheme for problem  \eqref{eqn:para_robin}. For the time discretization, we divide the interval $[0,T]$ into $M$ equally spaced subintervals with a time step size $\tau=T/M$ and take the time grids $\{t_m=m\tau\}_{m=0}^M$. We employ the backward Euler scheme to approximate the time derivative $\partial_tu$:
\begin{equation*}
    \bar{\partial}_t u(t_m)\approx \partial_\tau u^m=\tau^{-1}(u(t_m)-u(t_{m-1})).
\end{equation*}
For the spatial discretization, we employ the Galerkin FEM wth the finite element spaces $V_h$, $V_h^0$ and $V_h^{0,a}$ defined in Section \ref{subsec:FEM_approx}. 
Now we can describe the FEM approximation of problem \eqref{eqn:para_cts_functional}
\begin{equation}\label{eqn:para_dis_functional}
	\min_{q\in \mathcal{A}} J_{\al,h}(q)= \|\nabla u_{N,h}^M-\nabla u_{D,h}^\delta \|_{L^2(\Omega)}^2+\|  q^{\frac{1}{2}}(u_{N,h}^M-  u_{D,h}^\delta) \|_{L^2(\Gamma_i)}^2+\alpha\|  u_{N,h}^M-  u_{D,h}^\delta \|_{L^2(\Omega)}^2,
\end{equation}
where $ u_{N,h}^m=u_{N,h}^m(q) \in V_h$ solves
\begin{equation}\label{eqn:para_dis_constraint_NR}
	(\bar{\partial}_\tau  u_{N,h}^m , v_h)+(\nabla u_{N,h}^m,\nabla v_h)+(q u_{N,h}^m ,v_h)_{L^2(\Gamma_i)}=(g,v_h)_{L^2(\Gamma_a)},\quad \forall v_h\in V_h,\quad \mbox{with } u_{N,h}^0=R_h u_0,
\end{equation}
and $  u_{D,h}^\delta=u_{D,h}^\delta(q)\in V_h$  satisfies $u_{D,h}^\delta=z_h^\delta\mbox{ on }\Gamma_a'$ and
\begin{equation}\label{eqn:para_dis_constraint_DR}
	(\nabla u_{D,h}^\delta,\nabla v_h)+(q u_{D,h}^\delta,v_h)_{L^2(\Gamma_i)}
    = -(\bar{\partial}_\tau u_{N,h}^M, v_h)+ (g, v_h)_{L^2(\Gamma_a\setminus\Gamma_a')},\quad \forall v_h\in V_h^{0,a}.
\end{equation}
One can check that the discrete optimization problem \eqref{eqn:para_dis_functional}-\eqref{eqn:para_dis_constraint_DR} admits at least one global minimizer. Next we bound the error $\|q^{\dag}-q^*\|_{L^\infty(\Gamma_i)} $ between the exact Robin coefficient $q^{\dag}$ and the minimizer $q^*$. 

The following \textit{a priori} estimate on the discrete Neumann problem is crucial. The proof employs an operator theoretic argument \cite{Fujita:1991}, and is deferred to the appendix.
\begin{lem}\label{lem:para_error_L2_Neumann}
Let Assumption \ref{assum:regularity_para} hold, and $u(T;q^{\dag})$ and $u_{N,h}^M(q^\dag)$ be the solutions of problems \eqref{eqn:para_robin} and \eqref{eqn:para_dis_constraint_NR}, respectively.
Then the following error estimates hold
    \begin{align*}
       \|  u_{N,h}^M( q^{\dag})- u(T;q^{\dag}) \|_{L^2(\Omega)}&\le c(h^{2(1-\epsilon)}+\tau),\qquad \| \nabla u_{N,h}^M( q^{\dag})- \nabla u(T;q^{\dag}) \|_{L^2(\Omega)} \le c(h^{1-\epsilon}+T^{-\frac{1}{2}}\tau),\\
        \|  \bar{\partial}_\tau u_{N,h}^M( q^{\dag})- \partial_t u(T;q^{\dag}) \|_{L^2(\Omega)}& \le cT^{-1} (h^{2(1-\epsilon)}+\tau) .
    \end{align*}
\end{lem}

The following \textit{a priori} estimates on the data fitting term hold.
\begin{lem}\label{lem:para_error_L2_Dirichlet}
Let Assumption \ref{assum:regularity_para} hold, and $u(T;q^{\dag})$ and $u_{D,h}^\delta(  q^{\dag})$ be the solutions of problems \eqref{eqn:para_robin}  and \eqref{eqn:para_dis_constraint_DR}, respectively. Then  the following error estimates hold
    \begin{align*}
       \|  u_{D,h}^\delta( q^{\dag})- u(T;q^{\dag}) \|_{L^2(\Omega)}\le& c(h^{2(1-\epsilon)}+\delta^{\frac{4}{3}(1-\epsilon)}+ T^{-1} (h^{2(1-\epsilon)}+\tau)),\\
       \| \nabla u_{D,h}^\delta( q^{\dag})- \nabla u(T;q^{\dag}) \|_{L^2(\Omega)} \le&  c(h^{1-\epsilon}+\delta^{\frac{2}{3}(1-\epsilon)}+ T^{-1} (h^{2(1-\epsilon)}+\tau)).
    \end{align*}
\end{lem}
\begin{proof}
Let $\widetilde{u}_D\equiv \widetilde{u}_D(q^\dag)$ be the solution of the elliptic problem
    \begin{equation*}
    \left\{
        \begin{aligned}
            -\Delta \widetilde{u}_D & =\bar{\partial}_\tau u_{N,h}^M, && \mbox{in }\Om, \\
            \partial_n \widetilde{u}_D+q^\dag \widetilde{u}_D & =0, && \mbox{on } \Gamma_i,\\
            \partial_n \widetilde{u}_D & =g, && \mbox{on } \Gamma_a\setminus\Gamma_a',\\
            \widetilde{u}_D & =f, && \mbox{on } \Gamma_a'.
        \end{aligned}
    \right.
    \end{equation*}
By the triangle inequality, we have
\begin{equation*}
        \|  u_{D,h}^\delta( q^{\dag})- u(T;q^{\dag}) \|_{H^1(\Omega)}\le \|  u_{D,h}^\delta( q^{\dag})- \widetilde{u}_D \|_{H^1(\Omega)}+\|  \widetilde{u}_D- u(T;q^{\dag}) \|_{H^1(\Omega)}.
    \end{equation*}
The elliptic regularity, the definition \eqref{eqn:para_noise} and Lemma \ref{lem:para_error_L2_Neumann} imply
    \begin{align*}
        \|  \widetilde{u}_D( q^{\dag})- u(T;q^{\dag}) \|_{H^1(\Omega)}\le c\| \partial_t u(T;q^\dag)-\bar{\partial}_\tau u_{N,h}^M(q^\dag) \|_{L^2(\Omega)}
        \le  cT^{-1} (h^{2(1-\epsilon)}+\tau).
    \end{align*}
    The term $\|  u_{D,h}^\delta( q^{\dag})- \widetilde{u}_D(q^\dag) \|_{H^1(\Omega)} $ can be bounded similarly as Lemma \ref{lem:apriori}, with  an additional source $\bar{\partial}_\tau u_{N,h}^M\in L^2(\Omega) $. The regularity decomposition and the construction of the locally refined mesh $\mathcal{T}_h$ imply 
    \begin{equation*}
        \|  u_{D,h}^\delta( q^{\dag})- \widetilde{u}_D(q^\dag) \|_{H^1(\Omega)}\le c(h^{1-\epsilon}+\delta^{\frac{2}{3}(1-\epsilon)})\quad \text{and}\quad  \|  u_{D,h}^\delta( q^{\dag})- \widetilde{u}_D(q^\dag) \|_{L^2(\Omega)}\le c(h^{2(1-\epsilon)}+\delta^{\frac{4}{3}(1-\epsilon)}).
    \end{equation*}
Combining the preceding estimates completes the proof of the lemma.
\end{proof}

The following \textit{a priori} estimate on the objective functional $J_{\alpha,h}(q^*)$ holds.
\begin{lem}\label{lem:para_apriori}
Let Assumption \ref{assum:regularity_para} hold, and $q^*$ be a minimizer of problem \eqref{eqn:para_dis_functional}-\eqref{eqn:para_dis_constraint_DR}. Then  there holds
    \begin{equation*}
       J_{\alpha,h}(q^*)\le c(\tau^2+\eta+\alpha(\tau^2+\eta^2)),\quad \mbox{with }\eta = h^{2(1-\epsilon)}+\delta^{\frac{4}{3}(1-\epsilon)}.
    \end{equation*}
\end{lem}
\begin{proof}
Since $q^*$ minimizes the functional $J_{\alpha,h}$, we have $J_{\alpha,h}(q^*)\le J_{\alpha,h}( q^{\dag})$. Let $u^\dag=u(q^\dag)$ be the solution of problem \eqref{eqn:para_robin} with $q=q^\dag$. By Lemmas \ref{lem:para_error_L2_Neumann} and \ref{lem:para_error_L2_Dirichlet} and then taking $T$ large, we obtain
    \begin{align*}
        &\|\nabla (u_{N,h}^M(q^\dag)- u_{D,h}^\delta(q^\dag)) \|_{L^2(\Omega)}
        \le \|\nabla( u_{N,h}^M(q^\dag)- u^\dag(T)) \|_{L^2(\Omega)}+ \|\nabla(u^\dag(T)- u_{D,h}^\delta(q^\dag)) \|_{L^2(\Omega)}\\
        \le& c(h^{1-\epsilon}+T^{-\frac{1}{2}}\tau)+c(h^{1-\epsilon}+\delta^{\frac{2}{3}(1-\epsilon)}+ T^{-1} (h^{2(1-\epsilon)}+\tau)) 
        \le c(\eta^\frac12+\tau) ,\\
        &\| u_{N,h}^M(q^\dag)- u_{D,h}^\delta(q^\dag) \|_{L^2(\Omega)}
        \le  \| u_{N,h}^M(q^\dag)- u^\dag(T) \|_{L^2(\Omega)}+ \| u^\dag(T)- u_{D,h}^\delta(q^\dag) \|_{L^2(\Omega)}\\
        \le& c(h^{2(1-\epsilon)}+\tau)+c(h^{2(1-\epsilon)}+\delta^{\frac{4}{3}(1-\epsilon)}+ T^{-1} (h^{2(1-\epsilon)}+\tau))
        \le  c(\eta+\tau).
    \end{align*}
    Furthermore, the box constraint and Lemma \ref{lem:interp-ineq} imply
    \begin{align*}
        &\|  q^{\frac{1}{2}}(u_{N,h}^M(q^\dag)-  u_{D,h}^\delta(q^\dag)) \|_{L^2(\Gamma_i)}
        \le c\|   u_{N,h}^M(q^\dag)-  u_{D,h}^\delta(q^\dag) \|_{L^2(\Gamma_i)}\\
        \le & c\|   u_{N,h}^M(q^\dag)-  u_{D,h}^\delta(q^\dag) \|_{L^2(\Omega)}^{\frac{1}{2}}\|   u_{N,h}^M(q^\dag)-  u_{D,h}^\delta(q^\dag) \|_{H^1(\Omega)}^{\frac{1}{2}}\\
        \le & c( h^{\frac{3}{2}(1-\epsilon)}+\delta^{1-\epsilon}+\tau+h^{\frac{1}{2}(1-\epsilon)}\tau^{\frac{1}{2}}+\delta^{\frac{1}{3}(1-\epsilon)}\tau^{\frac{1}{2}} ) \leq c(\eta^\frac34 + \tau).
    \end{align*}
These estimates together show the desired assertion.
\end{proof}

\begin{thm}\label{thm:para_error_estimate}
Let Assumption \ref{assum:regularity_para} hold, and $q^*$ be a minimizer of problem \eqref{eqn:para_dis_functional}-\eqref{eqn:para_dis_constraint_DR}. Let $\eta = h^{2(1-\epsilon)}+\delta^{\frac43(1-\epsilon)}$. 
\begin{itemize}
    \item[{\rm(i)}] If $q^\dag\in\mathcal{A}$, then
    there exists $c>0$ independent of $h$, $\tau$, $\delta$ and $\alpha$ such that 
    $$ \| q^*-q^{\dag} \|_{L^\infty(\Gamma_i)} \le c((\alpha^{-\frac12} + \alpha^{-\frac14})(\tau+\eta^\frac12+\alpha^\frac12(\tau+\eta))+\eta^\frac{3}{4}+\tau^\frac34).$$
    \item[{\rm(ii)}] If $q^\dag\in\mathcal{B}$, then for $T\rightarrow\infty$, there exist $c>0$ independent of $h,\tau,\delta,\alpha$, $\theta=\frac{1-2\epsilon}{3-2\epsilon} $ with $\epsilon\in (0,\frac{1}{2})$, and $\kappa\equiv \kappa(c_g,\underline{c}_q, \Bar{c}_q, c_0,\Omega) \in(0,1)$ such that
       $$ \| q^*-q^{\dag} \|_{L^\infty(\Gamma_i)} \le c( (\alpha^{-\frac12} + \alpha^{-\frac14})(\tau+\eta^\frac12+\alpha^\frac12(\tau+\eta))+\eta^\frac{3}{4}+\tau^\frac34) ^{\kappa\theta}.$$
\end{itemize}
\end{thm}
\begin{proof}
In case (i), by Theorem \ref{thm:stab_para} (i), for sufficiently large $T$, we have
\begin{equation*}
  \| q^*-q^{\dag} \|_{L^\infty(\Gamma_i)} \le c\| u(T;q^*)-u(T;q^{\dag}) \|_{L^2(\Gamma_a')}.
\end{equation*}
The triangle inequality implies
\begin{align*}
        \|u(T;q^*)-u(T;q^{\dag}) \|_{L^2(\Gamma_a')}
        \le & \|u(T;q^*)-u_{N,h}^M(q^*)  \|_{L^2(\Gamma_a')} + \|u_{N,h}^M(q^*)-z_h^\delta  \|_{L^2(\Gamma_a')}\\
         &+\|z_h^\delta -u(T;q^{\dag}) \|_{L^2(\Gamma_a')}:={\rm I}+{\rm II} + {\rm III}.
    \end{align*}
By Lemmas \ref{lem:interp-ineq} and \ref{lem:para_error_L2_Neumann}, we arrive at
    \begin{align*}
        {\rm I} \le& c \|u(T;q^*)-u_{N,h}^M(q^*) \|_{L^2(\Omega)}^{\frac{1}{2}}\|u(T;q^*)-u_{N,h}^M(q^*) \|_{H^1(\Omega)}^{\frac{1}{2}}\\
        \le& c(h^{\frac{3}{2}(1-\epsilon)}+h^{\frac{1}{2}(1-\epsilon)}\tau^{\frac{1}{2}}+\tau) \le c(h^{\frac{3}{2}(1-\epsilon)} +\tau^{\frac{3}{4}}).
    \end{align*}
Similarly, Lemma \ref{lem:interp-ineq}  implies
\begin{align*}
 {\rm II}  =&\|u_{N,h}^M(q^*)-u_{D,h}^\delta(q^*)  \|_{L^2(\Gamma_a')}
 \le c \|u_{N,h}^M(q^*)-u_{D,h}^\delta(q^*)  \|_{L^2(\Omega)}^{\frac{1}{2}} \|u_{N,h}^M(q^*)-u_{D,h}^\delta(q^*)  \|_{H^1(\Omega)}^{\frac{1}{2}}\\
 \le & c(\alpha^{-1}J_{\alpha,h})^{\frac{1}{4}} (J_{\alpha,h}(q^*)+\alpha^{-1}J_{\alpha,h}(q^*))^{\frac{1}{4}}
 \leq  c(\alpha^{-\frac14}+\alpha^{-\frac12})J_{\alpha,h}(q^*)^\frac{1}{2}. 
\end{align*}
By Lemma  \ref{lem:para_apriori},
$  J_{\alpha,h}(q^*)^\frac12\le c(\tau+\eta^\frac12+\alpha^\frac12(\tau+\eta))$. Consequently,
\begin{equation*}
    {\rm II} \leq c(\alpha^{-\frac12} + \alpha^{-\frac14})(\tau+\eta^\frac12+\alpha^\frac12(\tau+\eta)).
\end{equation*}
The definition \eqref{eqn:para_noise} indicates
    $ {\rm III} \le c\delta.$
Combining the preceding estimates yields the assertion in case (i).\\
In case (ii), by Theorem \ref{thm:stab_para}(ii), for $r\in (\frac{1}{4},\frac{1}{2})$, $\theta=\frac{1-2\epsilon}{3-2\epsilon}$ and $\kappa\equiv
 \kappa(c_g, \underline{c}_q, \Bar{c}_q, c_0, \Omega)\in(0,1)$, we have
    \begin{equation*}
         \| q^*-q^{\dag} \|_{L^\infty(\Gamma_i)} \le c\big(\| u(T;q^*)-u(T;q^{\dag})\|_{L^2(\Gamma_a')}  + T^{ \frac{1}{2}-r }e^{-\min(\frac{\lambda(q^\dag)}{2},\frac{\lambda(q^*)}{2})T} \|q^*-q^\dag\|_{L^\infty(\Gamma_i)}  \big)^{\kappa\theta}.
    \end{equation*}
    By the triangle inequality, Lemma  \ref{lem:para_apriori} and definition \eqref{eqn:para_noise}, we obtain
    \begin{align*}
        \|u(T;q^*)-u(T;q^{\dag}) \|_{L^2(\Gamma_a')}
        \le &\|u(T;q^*)-u_{N,h}^M(q^*) \|_{L^2(\Gamma_a')} + \|u_{N,h}^M(q^*)-  z_h^\delta \|_{L^2(\Gamma_a')} +\|  z_h^\delta-u(T;q^{\dag}) \|_{L^2(\Gamma_a')}.
    \end{align*}
    Then combining the preceding estimates with $T\rightarrow\infty$ yields the assertion in case (ii).
\end{proof}
\begin{remark}\label{rem:conv_rate_para}
Theorem \ref{thm:para_error_estimate} provides a guideline for the \textit{a priori} selection of algorithmic parameters. The choice $h\sim \delta^{\frac{2}{3}}$, $\tau\sim \delta^{\frac{4}{3}}$ and $\alpha\sim \delta^{-\frac{4}{3}}$ yields a convergence rate $\| q_h^*-q^{\dag} \|_{L^\infty(\Gamma_i)} \le C\delta $.
\end{remark}

\section{Numerical experiments and discussions}\label{sec:numer}

Now we present numerical results to illustrate the proposed reconstruction algorithm. Throughout we take a square domain $\Omega=(-1,1)^2$ with the inaccessible part $\Gamma_i=\{1\}\times(-1,1)$ (i.e., the left edge), and unless otherwise stated, also take $\Gamma_a'=\Gamma_a$. We generate the exact data $f=u(q^\dag)|_{\Gamma_a'}$ using a much finer mesh, and  the noisy data $z^\delta$ by
\begin{equation*}
    z^\delta(x)=u^\dag(x)+\delta\|u^\dag\|_{L^\infty(\Gamma_a')} \xi(x) \quad\text{and}\quad z^\delta(x)=u^\dag(x,T)+\delta\|u^\dag(\cdot,T)\|_{L^\infty(\Gamma_a')} \xi(x),
\end{equation*}
where $\delta$ is the relative noise level, and $\xi $ follows the standard Gaussian distribution. We define $z_h^\delta=P_{h,\Gamma_a'} z^\delta$ as in Remark \ref{rem:noise} such that the assumptions \eqref{eqn:noise} and \eqref{eqn:para_noise} hold. The accuracy of a reconstruction $q^*$ relative to the exact one $q^\dag$ is measured by the relative $L^2(\Gamma_i)$ error $e_q=\|q^*-q^\dag\|_{L^2(\Gamma_i)}/\|q^\dag\|_{L^2(\Gamma_i)}$.

\subsection{The reconstructions with a known partition}
The first two examples are about $q^\dag\in \mathcal{A}$, i.e., with a known partition. The related optimization problems are minimized by the standard conjugate gradient method \cite{Jin:2007,Jin:2010} with an initial guess $q_0\equiv1$. The Gateaux derivative $J'_\alpha(q)$ can be computed using the adjoint technique, and the explicit formulas in the elliptic and parabolic cases are given in Appendix \ref{app:deriv}. For any point $x\in\Omega\subset\mathbb{R}^2$, we write $x=(x_1,x_2)$.
\begin{example}[Elliptic case]\label{ex:ell}
Take the boundary illumination $g(x_1,x_2)=1-x_1^2$, and consider the following three cases: {\rm(i)} $q^\dag(x_2)=1+\chi_{[0.2,1]}(x_2)$; {\rm(ii)} $q^\dag(x_2)=1+\chi_{[0.2,1]}(x_2)$, with $\Gamma_a'=\{1\}\times(-1,1)$; {\rm(iii)} $q^\dag(x_2)=1+\chi_{[0.2,0.6]}(x_2)$.
\end{example}

The relative errors are shown in Table \ref{tab:fixed_partition}, and the reconstructions in Fig. \ref{fig:reconstruction_fixed}, where in view of Remark \ref{rem:conv_rate_ell}, we choose the mesh size $h\sim \delta^{\frac{2}{3}}$ and the parameter $\alpha\sim \delta^{-\frac{4}{3}}$. In the experiment, we take a relatively small $\alpha$ so that the $L^2(\Omega)$ matching term in $J_\alpha$ does not dominate. For case (i), the numerical results indicate that the error $e_q$ decays to zero at a rate $O(\delta^{1.03})$ as $\delta\to0$, which agrees very well with the predicted rate $O(\delta)$ in Remark \ref{rem:conv_rate_ell}. We obtain highly accurate recoveries for up to $1\%$ noise level, cf. Fig. \ref{fig:reconstruction_fixed}. In case (ii), we still observe an $O(\delta^{1.01})$ convergence rate, indicating that partial measurement has little influence on the Lipschitz stability of the inverse Robin problem. However, the  accuracy of the reconstruction $q^*$ is lower than that for the full data on $\Gamma_a$ in case (i). In case (iii), the convergence rate is $O(\delta^{1.02})$, which agrees well with the theoretical estimate. By \cite[Theorem 2.4]{Sincich:2007} the Lipschitz constant increases exponentially with the number $N$ of pieces in the partition. Hence, compared with case (i), the reconstruction problem in case (iii) is far more challenging. Fig. \ref{fig:reconstruction_fixed} shows that the accuracy of the recovered coefficient $q^*$ deteriorates but is still reasonable.

\begin{table}[htb!]
  \centering\setlength{\tabcolsep}{3pt}
    \caption{The convergence rates for Examples \ref{ex:ell} and \ref{ex:para} with respect to $\delta$. $h_0$, $\tau_0$ and $\alpha_0$ denotes the initial mesh size, time step size and penalty parameter.}\label{tab:fixed_partition}
    \begin{tabular}{c|cccccc|cccccc|}
    \toprule
    \multicolumn{1}{c}{}&
    \multicolumn{6}{c}{(a) \ref{ex:ell}(i): $h_0=1/4$, $\alpha_0=1.6e\text{-4}$ }&\multicolumn{6}{c}{(b) \ref{ex:ell}(ii): $h_0=1/4$, $\alpha_0=1.6e\text{-4}$ }\\
    \cmidrule(lr){2-7} \cmidrule(lr){8-13}
    $\delta$   & 1e-2 & 5e-3 & 2e-3 & 1e-3 &5e-4 & \text{rate}  & 1e-2 & 5e-3 & 2e-3 & 1e-3 &5e-4 & \text{rate}\\
    \midrule
    $e_q$ & 8.70e-3 & 4.28e-3  &  1.51e-3 &  8.34e-4  &  4.00e-4 &$O(\delta^{1.03})$  &  8.88e-2  &  5.18e-2 & 2.21e-2  & 1.00e-2  &  4.26e-3 &$O(\delta^{1.01})$  \\
    \midrule
    \multicolumn{1}{c}{}&
    \multicolumn{6}{c}{(c) \ref{ex:ell}(iii): $h_0=1/8$, $\alpha_0=6.4e\text{-4}$ }&\multicolumn{6}{c}{(d) \ref{ex:para}(i): $h_0=1/4$, $\tau_0=1/5$, $\alpha=$1.6e-5 }\\
    \cmidrule(lr){2-7} \cmidrule(lr){8-13}
    $\delta$  & 1e-2 & 5e-3 & 2e-3 & 1e-3 &5e-4 & \text{rate}  & 1e-2 & 5e-3 & 2e-3 & 1e-3 &5e-4 & \text{rate}  \\
    \midrule
    $e_q$ & 1.98e-2  &  1.14e-2 & 3.34e-3  & 1.83e-3  & 1.06e-3  &$O(\delta^{1.02})$ & 6.52e-3 & 3.96e-3  &  1.57e-3 & 8.45e-4  &  3.73e-4 &$O(\delta^{0.96})$ \\
    \midrule
    \multicolumn{1}{c}{}&
    \multicolumn{6}{c}{(e) \ref{ex:para}(ii): $h_0=1/6$, $\tau_0=1/5$ $\alpha_0=$3.6e-5}\\
    \cmidrule(lr){2-7}
    $\delta$  & 1e-2 & 5e-3 & 2e-3 & 1e-3 &5e-4 & \text{rate}  \\
    \cmidrule(lr){1-7}
    $e_q$ &  4.82e-2 & 2.72e-2  &  9.77e-3 & 5.14e-3  &  2.14e-3 &$O(\delta^{1.04})$ \\
    \bottomrule
    \end{tabular}
\end{table}

\begin{figure}[hbt!]
    \centering
        \includegraphics[width=0.3\textwidth]{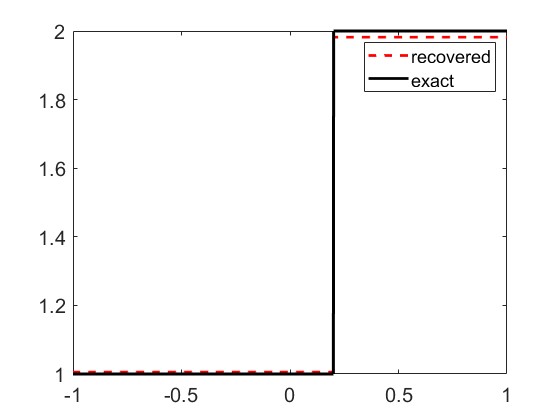}
        \includegraphics[width=0.3\textwidth]{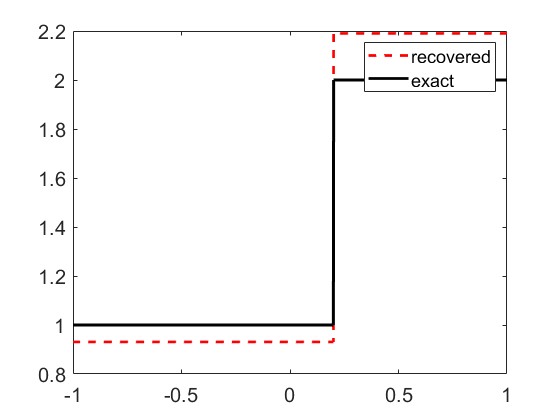}
        \includegraphics[width=0.3\textwidth]{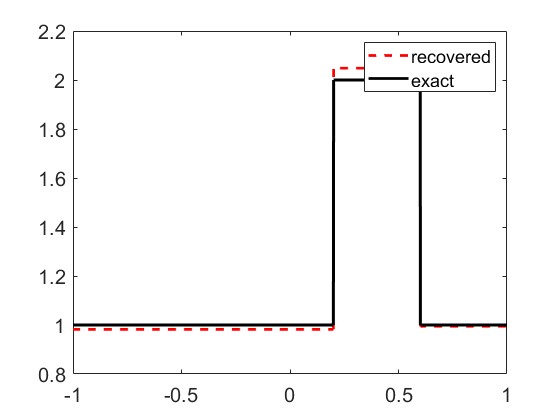}
        \includegraphics[width=0.3\textwidth]{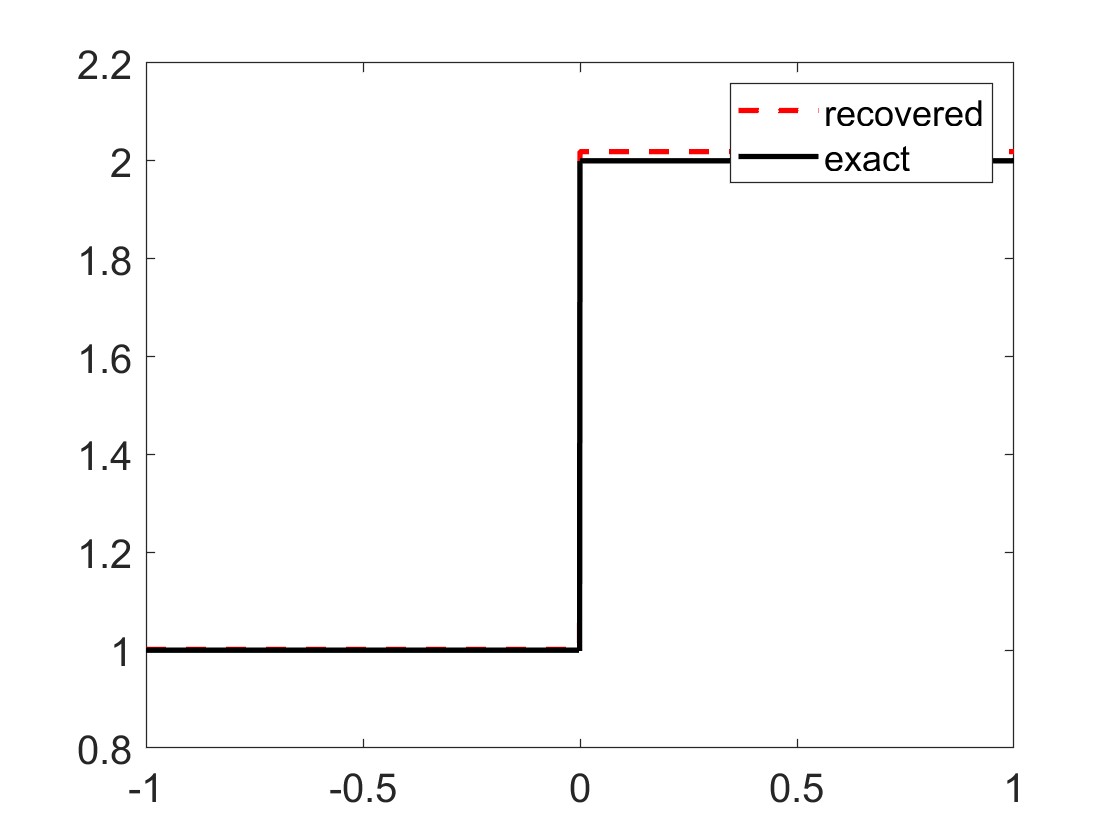}
        \includegraphics[width=0.3\textwidth]{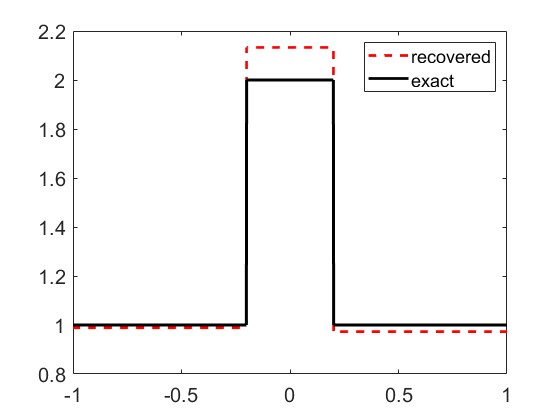}
\caption{The reconstruction of Examples \ref{ex:ell} and \ref{ex:para} with $\delta=1\%$ noise. Top row: reconstructions for Example \ref{ex:ell}(i)--(iii). Bottom row: reconstructions for Example \ref{ex:para}(i)--(ii). }
\label{fig:reconstruction_fixed}
\end{figure}

The next example is about the parabolic case. 
\begin{example}[Parabolic case]\label{ex:para}
Consider the time independent boundary illumination $g(x_1,x_2)=1-x_1^2$. To obtain an initial data $u_0$ satisfying Assumption \ref{assum:regularity_para}, let $u_0$ solve the following elliptic problem with the source $f(x_1,x_2)=\frac{1}{4}\cos(\pi x_1)\cos(\pi x_2)$:
\begin{equation*}
\left\{\begin{aligned}    -\Delta u_0&=f,\quad\mbox{in } \Omega, \\
\partial_nu_0+q^\dag u_0&=0,\quad\mbox{on } \Gamma_i, \\ \partial_nu_0&=g,\quad\mbox{on } \Gamma_a.
\end{aligned}\right.
\end{equation*}
The measurement is taken at the terminal time $T=10$ on the full part $\Gamma_a$. Consider the following two cases: {\rm(i)} $q(x_2)=1+ \chi_{[0,1]}(x_2)$; {\rm(ii)} $q(x_2)=1+ \chi_{[-0.2,0.2]}(x_2)$.
\end{example}

Table \ref{tab:fixed_partition} and Fig. \ref{fig:reconstruction_fixed} show the relative errors $e_q$ and reconstructions $q^*$, respectively. We choose the mesh size $h\sim\delta^{\frac{2}{3}}$, time step $\tau\sim\delta^{\frac{4}{3}} $ and parameter $\alpha\sim \delta^{-\frac{4}{3}}$ following Remark \ref{rem:conv_rate_para}. In case (i), the relative error $e_q$ exhibits an $O(\delta^{0.96})$ convergence rate, agreeing well with the prediction by Theorem \ref{thm:para_error_estimate}(i). The reconstruction is accurate  for up to $1\%$ noise level. In case (ii), the convergence rate is $O(\delta^{1.04})$, which is comparable with the prediction, and the recovered coefficient $q^*$ does not achieve very high accuracy, similar to case (iii) of Example \ref{ex:ell}.

\subsection{The reconstructions with an unknown partition}
Now we consider the case $q^\dag \in \mathcal{B}$, that is, the Robin coefficient is piecewise constant on an unknown partition. The discontinuity locations and the magnitude of the Robin coefficient $q$ have different influence on the observational data $z^\delta$, which make the objective $J_\alpha$ more challenging to optimize. Thus we use the gradient descent method to solve the optimization problem. Moreover, in order to get reliable reconstruction results, we fix the value of the homogeneous background at the exact one during the optimization process. First we consider the elliptic case. 
\begin{example}[Elliptic case]\label{ex:unknown_ell}
Take the boundary illumination $g(x_1,x_2)=1-x_1^2$, and consider the following three cases, with the initial guess $q_0$:
{\rm(i)} $q^\dag(x_2)=1+\chi_{[-0.3,0.3]}(x_2)$, with $q_0(x_2)=1+0.5\chi_{[-0.4,0.1]}(x_2) $, $\Gamma_a'=\{1\}\times(-1,1)$;
{\rm (ii)} $q^\dag(x_2)=1+\chi_{[-0.6,0.2]}(x_2)+1.5\chi_{[0.2,0.6]}(x_2)$, with $q_0(x_2)=1+0.5\chi_{[-0.5,-0.3]}(x_2)+0.5\chi_{[ 0.3, 0.5]}(x_2) $; {\rm (iii)} $q^\dag(x_2)=1+\chi_{[-0.6,0.6]}(x_2)$, with $q_0(x_2)=1+0.5\chi_{[-0.4,-0.1]}(x_2)+0.1\chi_{(-0.1,0.3]}(x_2)  $.
\end{example}

By Theorem \ref{thm:error_estimate} (ii), the relative error $e_q$ exhibits a H\"older rate, but the H\"older index $\kappa$ is implicit. We examine the convergence rate for Example \ref{ex:unknown_ell}(i). Table \ref{tab:unknown_partition} shows that the relative error $e_q$ decays at a rate $O(\delta^{0.50})$, in which the mesh size $h$ is chosen as $h\sim \delta^{\frac{2}{3}}$ and the penalty parameter $\alpha\sim \delta^{-\frac{4}{3}}$, as suggested by Remark \ref{rem:conv_rate_ell}.  Fig. \ref{fig:reconstruction_unknown} shows the optimization process and reconstruction results (with the choice $h=1/20$ and $\alpha=$4e-4). In case (i), the initial guess $q_0$ is located partially outside the exact inhomogeneity. The objective $J_{\alpha,h}$ decreases rapidly during the first few hundreds of iterations. The relative error $e_q$ first decreases and then increases a little bit, but it will decrease for further iterations. To speed up the convergence, we have increased the range of line search in the gradient descent method after 1000 iterations. Finally, after about 4000 steps, the relative error $e_q$ become nearly flat with small oscillations, indicating the convergence of the algorithm. We obtain a reasonable reconstruction with a relative error $e_q=$5.47e-2. The behavior of the objective value $J_{\alpha,h}(q^k)$ and error $e_q(q^k)$ show the challenges of reconstructing a piecewise constant Robin coefficient on an unknown partition.
Case (ii) is about reconstructing two inhomogeneities with the initial guess located inside. This choice leads to the fast decay of the objective value $J_{\alpha,h}(q^k)$ and relative error $e_q$. The recovered Robin coefficient $q^*$ can capture both the location and magnitude with a relative error $e_q=$9.74e-2.
In case (iii), the exact Robin coefficient $q^\dag$ has three portions while we assume there are four portions. The function value $J_{\alpha,h}(q^k)$ and relative error $e_q$ reach convergence in tens of iterations, and the reconstruction has a relative error $e_q=$6.75e-2.

\begin{table}[htp!]
  \centering
  \setlength{\tabcolsep}{3pt}
    \caption{The convergence rates for Example \ref{ex:unknown_ell}(i) and Example \ref{ex:unknown_para}(i) with respect to $\delta$. $h_0$, $\tau_0$ and $\alpha_0$ denote the initial mesh size, time step size and penalty parameter.}\label{tab:unknown_partition}
    \begin{tabular}{c|cccccc|cccccc|}
    \toprule
    \multicolumn{1}{c}{}&
    \multicolumn{6}{c}{(a) \ref{ex:unknown_ell}(i): $h_0=1/4$, $\alpha_0=$1.6e-5 }&\multicolumn{6}{c}{(b) \ref{ex:unknown_para}(i): $h_0=1/4$, $\tau_0=1/5$, $\alpha_=$1.6e-5 }\\
    \cmidrule(lr){2-7} \cmidrule(lr){8-13}
    $\delta$   & 1e-2 & 5e-3 & 2e-3 & 1e-3 &5e-4 & \text{rate}   & 1e-2 & 5e-3 & 2e-3 & 1e-3 &5e-4 & \text{rate}\\
    \midrule
    $e_q$ & 6.82e-2 & 5.45e-2  &  4.07e-2 &  2.27e-2  &  1.54e-2 &$O(\delta^{0.50})$  &  8.75e-2   & 7.93e-2   & 7.05e-2 &  6.01e-2  &  5.46e-2  & $O(\delta^{0.16})$ \\
    \bottomrule
    \end{tabular}
\end{table}
\begin{figure}
    \centering
    \begin{tabular}{ccc}
    \includegraphics[width=0.33\textwidth]{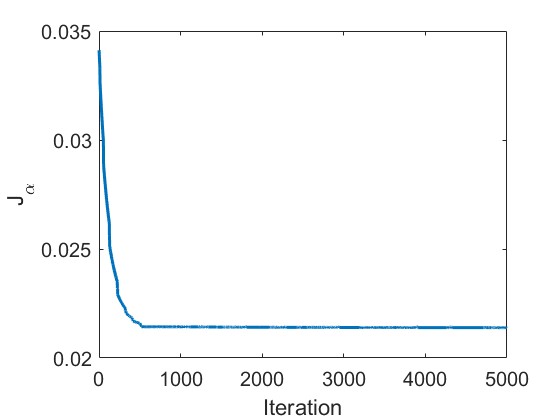}
    &    \includegraphics[width=0.33\textwidth]{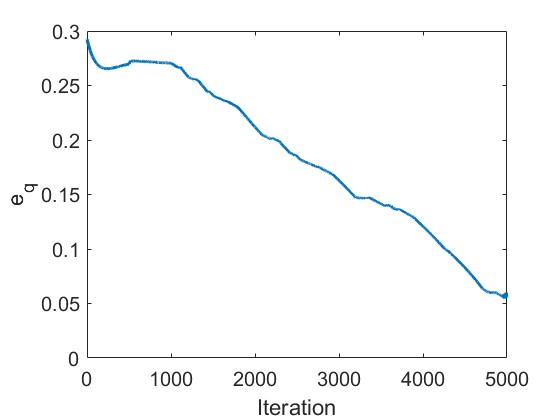}
     &   \includegraphics[width=0.33\textwidth]{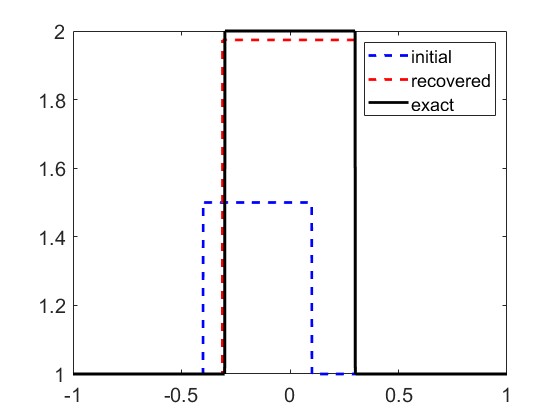}  \\
   \includegraphics[width=0.33\textwidth]{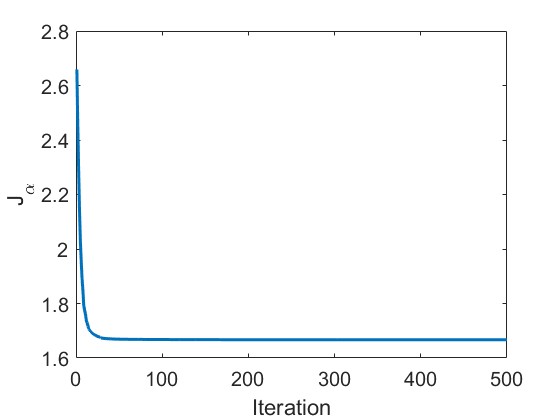}
    &    \includegraphics[width=0.33\textwidth]{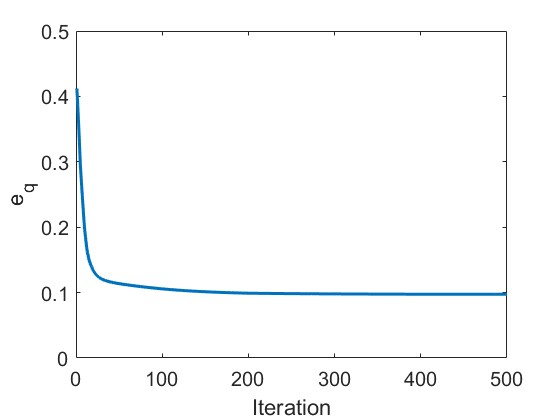}
    &    \includegraphics[width=0.33\textwidth]{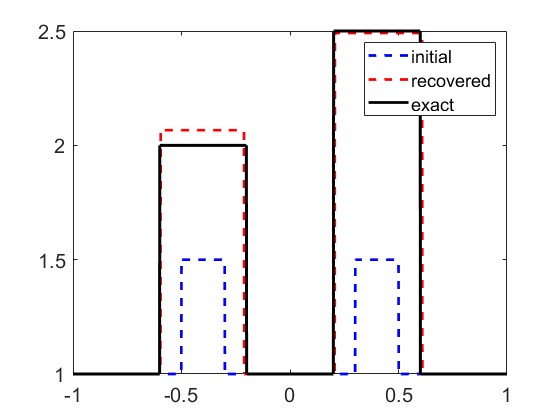}    \\
        \includegraphics[width=0.33\textwidth]{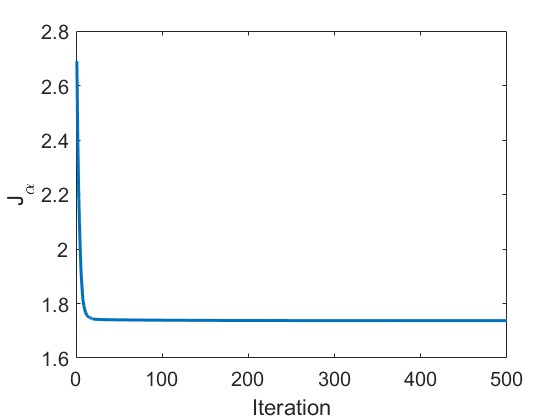}
& \includegraphics[width=0.33\textwidth]{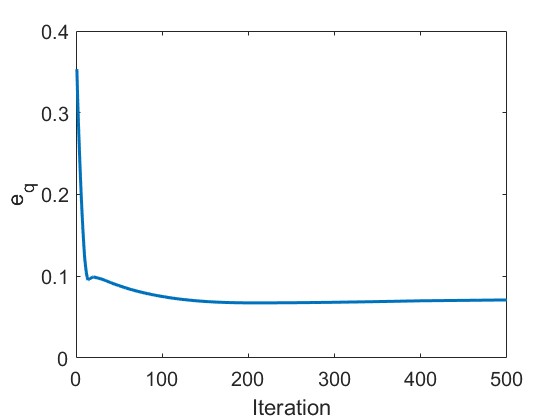}
& \includegraphics[width=0.33\textwidth]{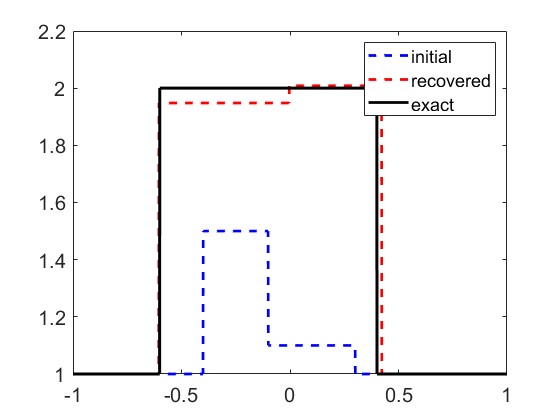}  \\      \includegraphics[width=0.33\textwidth]{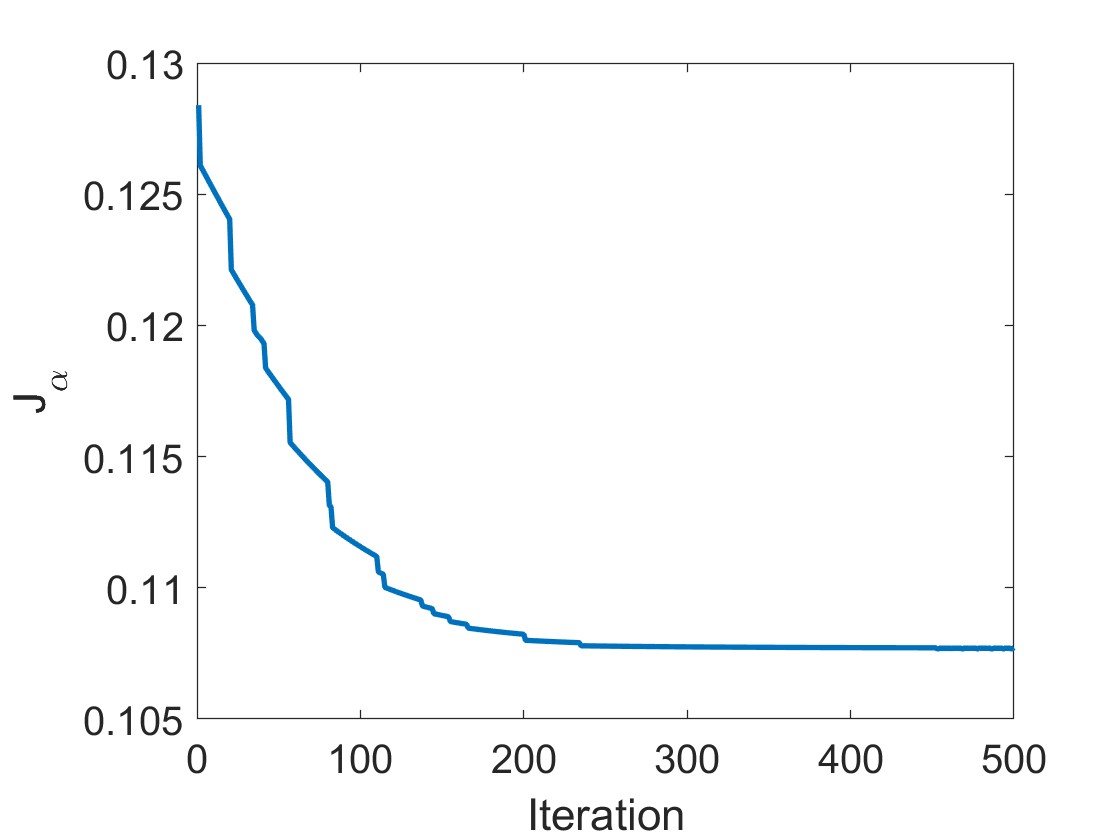}
&\includegraphics[width=0.33\textwidth]{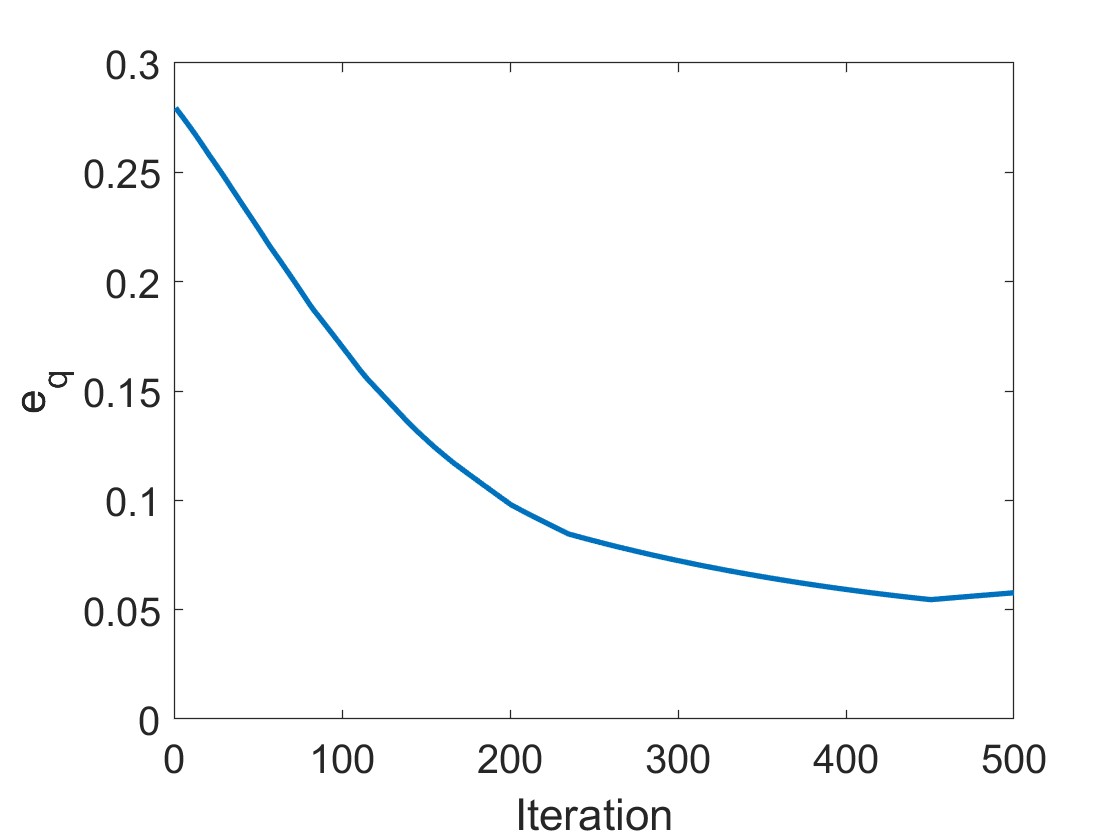}
&\includegraphics[width=0.33\textwidth]{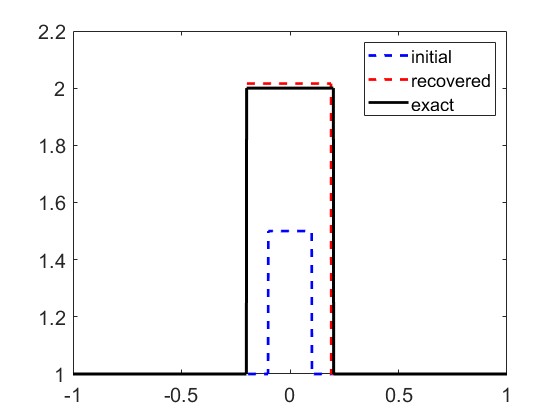}   \\
\includegraphics[width=0.33\textwidth]{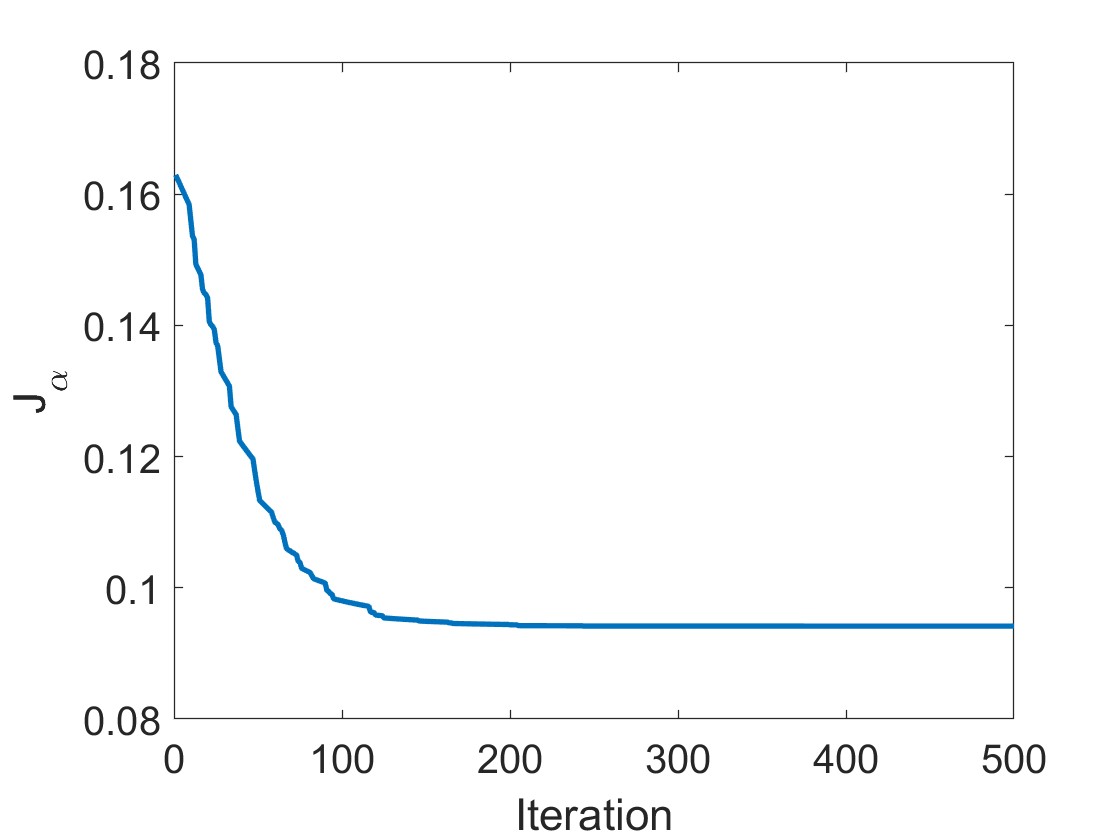}
&\includegraphics[width=0.33\textwidth]{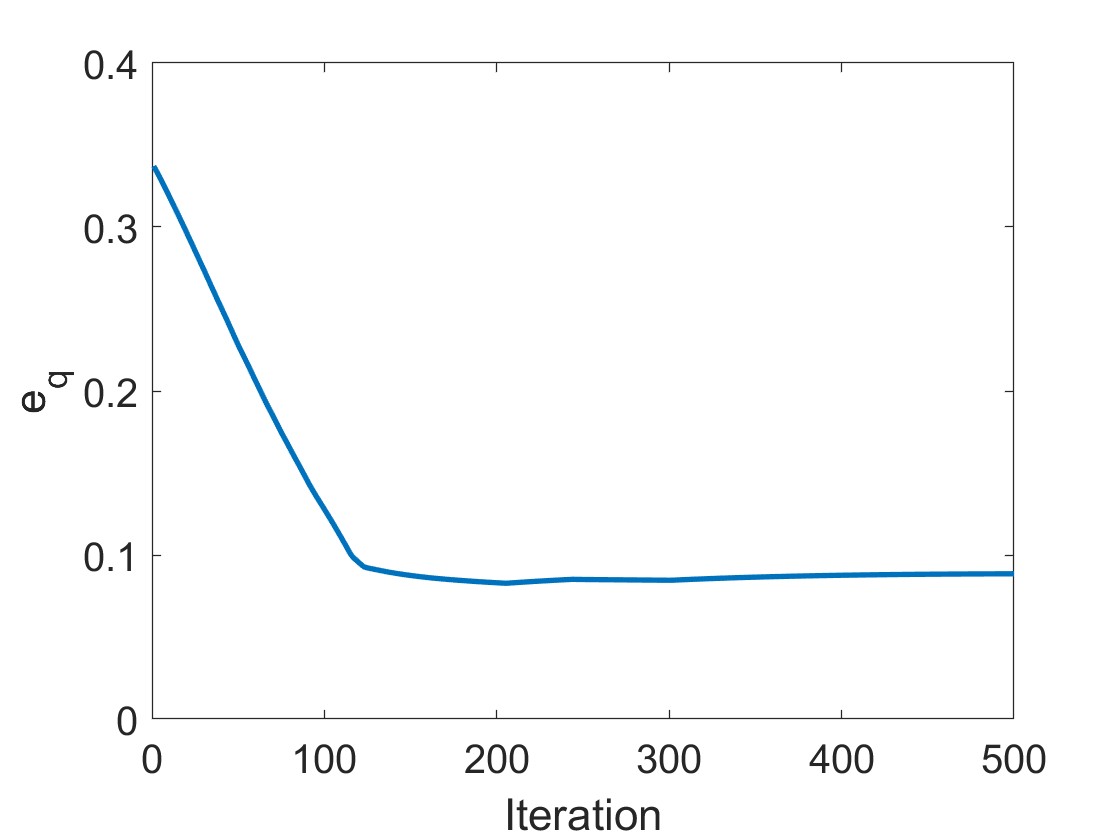}
&\includegraphics[width=0.33\textwidth]{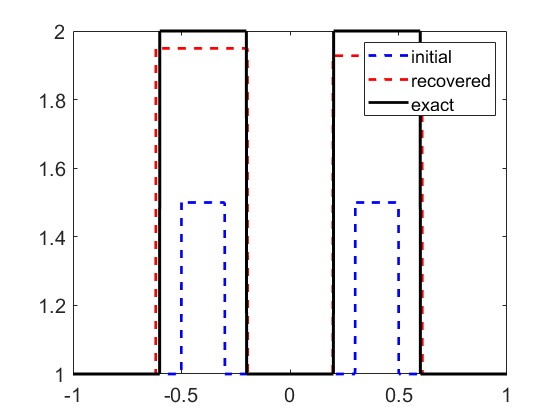}\\
(a) $J_{\alpha,h}(q^k)$ versus $k$ & (b) evolution of $e_q(q^k)$ versus $k$ & (c)  reconstruction $q^*$
\end{tabular}
 \caption{The optimization dynamics and reconstructions for Examples \ref{ex:unknown_ell} and \ref{ex:unknown_para} with noise level $1\%$.}
\label{fig:reconstruction_unknown}
\end{figure}

The last example shows the reconstructions in the parabolic case.
\begin{example}[Parabolic case]\label{ex:unknown_para}
 The boundary data $g$ and initial data $u_0$ are identical with that in Example \ref{ex:para}. The measurement is taken at the terminal time $T=10$ on the full part $\Gamma_a$. Consider the following two cases:
{\rm(i)} $q^\dag(x_2)=1+\chi_{[-0.2,0.2]}(x_2)$, with $q_0(x_2)=1+0.5\chi_{[-0.1, 0.1]}(x_2) $; {\rm(ii)} $q^\dag(x_2)=1+\chi_{[-0.6,-0.2]}(x_2)+\chi_{[0.2,0.6]}(x_2)$, with $q_0(x_2)=1+0.5\chi_{[-0.5,-0.3]}(x_2)+0.5\chi_{[0.3,0.5]}(x_2) $.
\end{example}

First we examine the convergence rate for Example \ref{ex:unknown_para}(i). In Table \ref{tab:unknown_partition}, we show the relative error $e_q$, where we take the parameters as $h\sim \delta^{\frac{2}{3}}$, $\tau\sim \delta^{\frac{4}{3}}$ and $\alpha\sim \delta^{-\frac{4}{3}}$.. The error $e_q$ is observed to decay at a rate $O(\delta^{0.16})$, which is much lower than the elliptic case in Example \ref{ex:unknown_ell}(i). This might be attributed to the choice of $T$, since Theorem \ref{thm:para_error_estimate}(ii) requires $T\rightarrow\infty$ in order to get the desired error bound.  Fig. \ref{fig:reconstruction_unknown} presents the optimization process and the reconstruction results (with the choice $h=1/20$, $\tau=1/100$ and $\alpha=$ 4e-4). The relative error $e_q$ is 5.46e-2 and 8.28e-2 for cases (i) and (ii), respectively. These results indicate that the approach also works reasonably well in the parabolic case, when the partition is unknown.

\section{Conclusion}
In this work we have investigated the numerical reconstruction of a piecewise constant Robin coefficient from the Cauchy data. We proposed a numerical approach based on the Kohn-Vogelius functional, discretized with the Galerkin FEM on specially graded meshes motivated by the solution decomposition. We have established error bounds on the discrete approximation that are explicit in terms of the noise level, regularization parameter and mesh size (and also the time step size in the parabolic case). We have also presented several numerical experiments illustrating the feasibility of the scheme. 
Extending the proposed scheme and its corresponding error analysis to the three-dimensional setting is a natural and important next step. However, this presents additional challenges, such as increased geometric complexity and the heightened difficulty of deriving suitable stability estimates. Another promising direction is investigating the inverse Robin problem in fluid dynamics, particularly for the Stokes system. In this context, the H\"older-type stability estimate remains largely unexplored, which presents an opportunity for significant advances.

\appendix
\section{Proof of Lemma \ref{lem:para_error_L2_Neumann}}

The proof employs an operator theoretic argument \cite{Fujita:1991}. 
First, we define the discrete Laplacian operator $A_h:V_h\rightarrow V_h$ by $ (A_h \phi_h,v_h)  =(\nabla \phi_h, \nabla v_h) + (q \phi_h, v_h)_{L^2(\Gamma_i)}$ for all $\phi_h, v_h\in V_h$. Then the solution $u_{N,h}^m$ to problem \eqref{eqn:para_dis_constraint_NR} can be represented as \cite[Section 3.1]{JinZhou:2023book}
    \begin{equation}\label{eqn:dis_sol_reps}
        u_{N,h}^m=R_h u_0-\frac{1}{2\pi \mathrm{i}}\int_{\Gamma_{\theta,\sigma}^\tau} e^{zt_{m}}e^{-z\tau } \delta_\tau(e^{-z\tau })^{-1}( \delta_\tau(e^{-z\tau }) +A_h)^{-1} A_h R_hu_0 \d z,
    \end{equation}
with the kernel function $\delta_\tau(\xi)=\frac{1-\xi}{\tau}$ and the contour $\Gamma_{\theta,\sigma}^{\tau}:=\{z\in\Gamma_{\theta,\sigma}: |\Im(z)|\le \pi/\tau \}$ with $\theta\in (\frac{\pi}{2},\pi)$ close to $\pi/2$ (oriented counterclockwise). Using standard energy argument \cite[p. 92]{Thomee:2007}, the following discrete resolvent estimates hold for all $z\in \Sigma_{\theta}:=\{0\ne z\in\mathbb{C}: |\arg(z)|\le \theta \}$:
\begin{equation}\label{eqn:dis_resolvent}
   \|(z+A_h)^{-1}\|_{L^2(\Omega)\rightarrow L^2(\Omega)}\le c\min(1,|z|^{-1})\quad \mbox{and}\quad
             \|(\delta_\tau(e^{-z\tau})+A_h)^{-1}\|_{L^2(\Omega)\rightarrow L^2(\Omega)}\le c\min(1,|z|^{-1}).
     \end{equation}

Now we  can bound the error. First, we bound the spatial discretization error. Note that $u(t)$ can be represented as
    \begin{equation*}
        u(t)=u_0-\frac{1}{2\pi \mathrm{i}}\int_{\Gamma_{\theta,\sigma}} e^{zt}z^{-1}(z+A)^{-1} Au_0 \d z.
    \end{equation*}
Further, let $u_{N,h}$ be the spatially semidiscrete solution given by
     \begin{equation*}
         u_{N,h}(t)=R_h u_0-\frac{1}{2\pi \mathrm{i}}\int_{\Gamma_{\theta,\sigma}} e^{zt}z^{-1}(z+A_h)^{-1} A_h R_hu_0 \d z.
     \end{equation*}
We first estimate the spatial discretization error $e_{N,h}:=u(T)-u_{N,h}(T)$. Let $K(z)=(z+A)^{-1}-(z+A_h)^{-1}P_h$, with $P_h$ denoting the $L^2(\Omega)$ projection on the the finite element space $V_h$. By \cite[Theorem 7.1]{Fujita:1991}, the following estimates hold
\begin{equation*}
         \|K(z) \|_{L^2(\Omega)\rightarrow H^s(\Omega)}\le ch^{(2-s)(1-\epsilon)}, \quad s=0,1.
     \end{equation*}
By the estimate \eqref{eqn:error_Rh}, the identity $A_hR_h=P_h A$ \cite[p. 11]{Thomee:2007}, we deduce that for $s=0,1$
\begin{align*}
\|e_{N,h}\|_{H^s(\Omega)} \le  &\|u_0-R_hu_0\|_{H^s(\Omega)}+C\int_{\Gamma_{\theta,\sigma}} |e^{zT}||z|^{-1}\|K(z) \|_{L^2(\Omega)\rightarrow H^s(\Omega)}\|Au_0\|_{L^2(\Omega)}\d |z|\\
   \le & ch^{(2-s)(1-\epsilon)}+ch^{(2-s)(1-\epsilon)}\left(\int_{T^{-1}}^{\infty} e^{-c\rho T}\rho^{-1}\d \rho +\int_{-\theta}^{\theta} e^{\cos\psi} \d \psi  \right)\le ch^{(2-s)(1-\epsilon)}.
\end{align*}
Let $e_{N,h}^M:=u_{N,h}(T)-u_{N,h}^M$ be the time discretization error. Note that $u_{N,h}^m$ is represented by \eqref{eqn:dis_sol_reps}. Hence, with $\widetilde{K}(z)=z^{-1}(z+A_h)^{-1}-e^{-z\tau } \delta_\tau(e^{-z\tau })^{-1}( \delta_\tau(e^{-z\tau }) +A_h)^{-1}$, 
\begin{align*}
         e_{N,h}^M
         =&\frac{1}{2\pi \mathrm{i}}\int_{\Gamma_{\theta,\sigma}\setminus\Gamma_{\theta,\sigma}^\tau} e^{zT}z^{-1}(z+A_h)^{-1} A_h R_hu_0 \d z
         +\frac{1}{2\pi \mathrm{i}}\int_{\Gamma_{\theta,\sigma}^\tau} e^{zT}\widetilde{K}(z) A_h R_hu_0 \d z:=\mathrm{I}+\mathrm{II}.
     \end{align*}
     By the discrete resolvent estimates in \eqref{eqn:dis_resolvent}, we have for $s=0,1$
\begin{align*}
\|\mathrm{I}\|_{H^s(\Omega)}
        \le& c\int_{\Gamma_{\theta,\sigma}\setminus\Gamma_{\theta,\sigma}^\tau} |e^{zT}||z|^{-1}\|(z+A_h)^{-1}\|_{L^2(\Omega)\rightarrow H^s(\Omega)} \|Au_0\|_{L^2(\Omega)} \d |z|
        \le c\int_{\tau^{-1}}^{\infty} e^{-c\rho T} \rho^{-(2-\frac{s}{2})} \d \rho\le cT^{-\frac{s}{2}}\tau.
\end{align*}
Meanwhile, by Taylor expansion, we deduce that for any $z\in \Gamma_{\theta,\sigma}^\tau$, the following estimates hold
    \begin{equation*}
        |1-e^{-z\tau}|\le c|z|\tau,\quad c|z|\le |\delta_\tau(e^{-z\tau})|\le c|z|,\quad |\delta_\tau(e^{-z\tau})-z|\le c\tau |z|^2.
    \end{equation*}
These estimates and the discrete resolvent estimates in \eqref{eqn:dis_resolvent} imply for $s=0,1$,
    \begin{align*}
         \|\widetilde{K}(z)\|_{L^2(\Omega)\rightarrow H^s(\Omega)}&\le  c\tau |z|^{-(1-\frac{s}{2})}.
    \end{align*}
Consequently,
    \begin{align*}
        \|\mathrm{II}\|_{L^2(\Omega)}\le C\tau \int_{\Gamma_{\theta,\sigma}^\tau} |e^{zT}| |z|^{-(1-\frac{s}{2})} \d |z|\le cT^{-\frac{s}{2}}\tau.
    \end{align*}
This proves the first statement. Using the solution representations
$\partial_t u(t)=-\frac{1}{2\pi \mathrm{i}}\int_{\Gamma_{\theta,\sigma}} e^{zt}(z+A)^{-1} Au_0 \d z$ and
$\bar{\partial}_\tau u_{N,h}^m=-\frac{1}{2\pi \mathrm{i}}\int_{\Gamma_{\theta,\sigma}^\tau} e^{zt_{m}}e^{-z\tau }  ( \delta_\tau(e^{-z\tau }) +A_h)^{-1} A_h R_hu_0 \d z$ and repeating the preceding argument yield
    \begin{equation*}
        \|  \bar{\partial}_\tau u_{N,h}^M( q^{\dag})- \partial_t u(T;q^{\dag}) \|_{L^2(\Omega)}\le cT^{-1} (h^{2(1-\epsilon)}+\tau).
    \end{equation*}
This completes the proof of the lemma.

\section{Gateaux derivative of the functional}\label{app:deriv}
The Gateaux derivatives of the Kohn-Vogelius functionals are given below.
\begin{lem}\label{lem:derivative}
The Gateaux derivative $J_\alpha'(q)$ of $J_\alpha$ in \eqref{eqn:cts_functional},with respect to $q$ is given by
\begin{equation*}
    J_\alpha'(q)=\left[u_D^2-u_N^2+2\alpha(u_Dv_D-u_Nv_N)]\right|_{\Gamma_i},
\end{equation*}
where $v_N$ and $v_D$ are respectively weak solutions to
\begin{equation*}
    \left\{
        \begin{aligned}
            -\Delta v_N&=u_N-u_D, &&\mbox{ in }\Omega,\\
            \partial_n v_N+ qv_N&=0, &&\mbox{ on } \Gamma_i,\\
            \partial_n v_N &=0, &&\mbox{ on } \Gamma_a,\\
        \end{aligned}
    \right.  \quad \mbox{and} \quad   \left\{
        \begin{aligned}
            -\Delta v_D&=u_N-u_D, &&\mbox{ in }\Omega,\\
            \partial_nv_D+ qv_D&=0, &&\mbox{ on } \Gamma_i,\\
            \partial_nv_D &=0, &&\mbox{ on } \Gamma_a\setminus\Gamma_a',\\
              v_D  &=0, &&\mbox{ on }  \Gamma_a'.\\
        \end{aligned}
    \right.
\end{equation*}
Similarly, for the functional $J_\alpha$ in \eqref{eqn:para_cts_functional}, the Gateaux derivative $J_\alpha'(q)$ of $J_\alpha$ is given by
\begin{equation*}
    \begin{aligned}
            J_\alpha'(q)=&\Bigg[u_D^2-u_N(T)^2+2\int_0^T \partial_t u_N(t) v_N(t) \d t\\
            &\qquad+2\alpha\left(u_Dv_D-\int_0^T u_N(t)v_N(t)\d t+\int_0^T \partial_t u_N(t) w(t) \d t \right)\Bigg]\Bigg|_{\Gamma_i},
    \end{aligned}
\end{equation*}
where $v_N$, $v_D$ and $w$ are solutions to
\begin{equation*}
    \left\{
        \begin{aligned}
            -\partial_t v_N-\Delta v_N&=(u_N-u_D)\delta_T(t), &&\mbox{ in }\Omega\times (0,T),\\
            \partial_n v_N+ qv_N&=0, &&\mbox{ on } \Gamma_i\times(0,T),\\
            \partial_n v_N &=0, &&\mbox{ on } \Gamma_a\times(0,T),\\
            v_N(T)&=0, &&\mbox{ in } \Omega,
        \end{aligned}
    \right.  \qquad   \left\{
        \begin{aligned}
            -\Delta v_D&=u_N-u_D, &&\mbox{ in }\Omega,\\
            \partial_n v_D+ qv_D&=0, &&\mbox{ on } \Gamma_i,\\
            \partial_n v_D &=0, &&\mbox{ on } \Gamma_a\setminus\Gamma_a',\\
              v_D  &=0, &&\mbox{ on }  \Gamma_a'.\\
        \end{aligned}
    \right.
\end{equation*}
\begin{equation*}
      \mbox{and}\quad  \left\{
        \begin{aligned}
            -\partial_t w-\Delta w&=v_D\delta_T(t), &&\mbox{ in }\Omega\times (0,T),\\
            \partial_n w+ qw&=0, &&\mbox{ on } \Gamma_i\times(0,T),\\
            \partial_nw &=0, &&\mbox{ on } \Gamma_a\times(0,T),\\
            w(T)&=0, &&\mbox{ in } \Omega,
        \end{aligned}
    \right.
\end{equation*}
where $\delta_T(\cdot)$ denotes the Dirac delta function at $T$.
\end{lem}
\begin{proof}
The result for the elliptic case can be found in \cite[Theorem 4]{Chaabane:1999}.  We prove only the parabolic case and $\Gamma_a'=\Gamma_a$. Let $J_\alpha(q)=J_{\rm KV}(q)+\alpha J_{\rm r}(q)$, with $J_{\rm KV}(q)= \|\nabla u_N(T)-\nabla u_D  \|_{L^2(\Omega)}^2+\|  q^{\frac{1}{2}}(u_N(T)-  u_D ) \|_{L^2(\Gamma_i)}^2$ and $  J_{\rm r}(q)=\|  u_N(T)-  u_D  \|_{L^2(\Omega)}^2$. Upon expansion, we have
\begin{align*}
        J_{\rm KV}(q)=&[\|\nabla u_N(T)\|_{L^2(\Omega)}^2 + (qu_N(T),u_N(T))_{L^2(\Gamma_i)}] + [\|\nabla u_D\|_{L^2(\Omega)}^2 + (qu_D,u_D)_{L^2(\Gamma_i)}]\\
        &[-2(\nabla u_N(T),\nabla u_D )-2(q u_N(T), u_D)_{L^2(\Gamma_i)}]:=J_N+J_D+J_{ND}.
    \end{align*}
Next we compute the Gateaux derivatives of $J_N$, $J_D$ and $J_{ND}$ separately.
The Gateaux derivative $J_N'(q)[p]$ of $J_N$ at $q$ in the direction $p$ is given by
    \begin{equation*}
        J_N'(q)[p]= 2( \nabla u_N(T) , \nabla \widetilde{u}_N(T)) +2(q  u_N(T) ,\widetilde{u}_N(T))_{L^2(\Gamma_i)} + (p, u_N(T)^2)_{L^2(\Gamma_i)},
    \end{equation*}
    where $\widetilde{u}_N$ is the directional derivative of $u_N(T)$ at $q $ along the direction $p$ which satisfies $\widetilde{u}_N(0)=0$ and
    \begin{equation*}
     (\partial_t \widetilde{u}_N(t),\varphi) +  (\nabla \widetilde{u}_N(t),\nabla\varphi)+ (q \widetilde{u}_N(t), \varphi )_{L^2(\Gamma_i)} =-( p{u}_N(t), \varphi)_{L^2(\Gamma_i)}, \quad \forall \varphi\in H^1(\Omega).
    \end{equation*}
    Similarly, the Gateaux derivative $J_D'(q)[p]$ is given by
    \begin{align*}
        J_D'(q)[p]&= 2 (\nabla u_D, \nabla \widetilde{u}_D) +2 (q  u_D   ,\widetilde{u}_D)_{L^2(\Gamma_i)}+(p, u_D^2)_{L^2(\Gamma_i)},
    \end{align*}
    with $\widetilde{u}_D\in H_{\Gamma_a}^1(\Omega):=\{v\in H^1(\Omega): v|_{\Gamma_a}=0\}$ satisfying  \begin{equation*}
      (\nabla \widetilde{u}_D ,\nabla\varphi) +(q \widetilde{u}_D, \varphi)_{L^2(\Gamma_i)} =-(p{u}_D, \varphi)_{L^2(\Gamma_i)}-( \partial_t \widetilde{u}_N(T), \varphi) , \quad \forall \varphi\in H_{\Gamma_a}^1(\Omega).
    \end{equation*}
    For the term $J_{ND}$, by integration by parts and the weak form of problem \eqref{eqn:para_robin}, we have
    \begin{align*}
        J_{ND}&=-2 (\partial_n u_N(T), u_D)_{L^2(\partial\Omega)} +2( \Delta u_N(T), u_D)-2( q u_N(T), u_D )_{L^2(\Gamma_i)}\\
        &=-2 (g, u_D)_{L^2(\Gamma_a)} +2( \partial_t u_N(T), u_D).
    \end{align*}
Since $u_D|_{\Gamma_a}=f$, the directional derivative $J_{ND}'(q)[p]$ of $J_{ND}$ is given by
    \begin{equation*}
         J_{ND}'(q)[p]=2 (\partial_t \widetilde{u}_N(T),u_D)+2( \partial_t u_N(T), \widetilde{u}_D).
    \end{equation*}
Combining the expressions of $J_{N}'(q)[p]$, $J_{D}'(q)[p]$ and $J_{ND}'(q)[p] $ and using the weak forms of $\widetilde{u}_N$ (with $\varphi=u_N$) and $ {u}_D$ (with $\varphi=\widetilde{u}_D$) lead to
    \begin{equation*}
        J_{\rm KV}'(q)[p]=(p,u_D^2-u_N(T)^2)_{L^2(\Gamma_i)}+2 ( \partial_t \widetilde{u}_N(T),u_D-u_N(T)).
    \end{equation*}
    To simplify the term $(\partial_t \widetilde{u}_N(T),u_D-u_N(T))$, we test the weak forms of $v_N$ and  $\partial_t \widetilde{u}_N$ with $\partial_t \widetilde{u}_N$ and $v_N$, respectively and apply integration by parts, and obtain
    \begin{equation*}
       (\partial_t \widetilde{u}_N(T),u_D-u_N(T))=\int_0^T (p\partial_t u_N(t), v_N(t))_{L^2(\Gamma_i)} \d t.
    \end{equation*}
Combining the preceding identities yields 
$$J_{\rm KV}'(q)=u_D^2-u_N(T)^2+2\int_0^T \partial_t u_N(t)  v_N(t)  \d t.$$
Next, we have
$J_{\rm r}'(q)[p]=2( \widetilde{u}_N(T)-\widetilde{u}_D,u_N(T)-u_D)$, 
where $\widetilde{u}_N(T)$and $\widetilde{u}_D $ are the directional derivatives of $u_N(T)$ and $u_D$, respectively. Taking the test function $ \widetilde{u}_N$ in the weak form of $v_N$ and taking test function $v_N$ in the weak form of $ \widetilde{u}_N$, we obtain
    \begin{equation*}
        ( \widetilde{u}_N(T) , u_N(T)-u_D) =-\int_0^T (p u_N(t),  v_N(t))_{L^2(\Gamma_i)} \d t.
    \end{equation*}
Similarly, using the weak forms of $\widetilde{u}_D$ and $v_D$, we have
    \begin{align*}
        (\widetilde{u}_D,u_N(T)-u_D)  =-(p u_D,  v_D)_{L^2(\Gamma_i)}-(\partial_t\widetilde{u}_N(T),v_D). 
    \end{align*}
For the term $(\partial_t\widetilde{u}_N(T),v_D)$, we test the weak forms of $w$ and $\partial_t \widetilde{u}_N$ by $\partial_t \widetilde{u}_N$ and $w$, respectively, and obtain
    \begin{align*}
        (\widetilde{u}_D,u_N(T)-u_D)  
        =&- (p u_D,v_D)_{L^2(\Gamma_i)}-\int_0^T (p \partial_t u_N(t), w(t))_{L^2(\Gamma_i)}  \d t.
    \end{align*}
Therefore, we obtain 
$$
J_{\rm r}'(q)=2\left[u_Dv_D-\int_0^T u_N(t)v_N(t)\d t+\int_0^T \partial_t u_N(t) w(t) \d t\right].$$
Combining these identities completes the proof of the proposition.
\end{proof}

\bibliographystyle{abbrv}

\end{document}